\documentclass{article}
\usepackage{geometry}
\newcommand{\keywords}[1]{{\bf Keywords.}~#1}
\newcommand{\subclass}[1]{{\bf MSC2010.}~#1}

\usepackage{xcolor}
\usepackage[textsize=footnotesize,backgroundcolor=yellow!70,bordercolor=orange]{todonotes}

\newcommand{\revision}[1]{\textcolor{black}{#1}}
\newcommand{\minor}[1]{\textcolor{black}{#1}}

\usepackage{graphicx}

\usepackage{amsmath,amssymb}
\usepackage{nicefrac}
\usepackage{amsfonts}
\usepackage{amsthm}
\newtheorem{definition}{Definition}
\newtheorem{remark}{Remark}
\newtheorem{proposition}{Proposition}
\newtheorem{lemma}{Lemma}

\usepackage{multirow}
\usepackage{pgfplotstable,booktabs}
\usepackage[caption=false]{subfig}
\allowdisplaybreaks

\newcommand{\CWENO}{\ensuremath{\mathsf{CWENO}}}
\newcommand{\CWENOZ}{\ensuremath{\mathsf{CWENOZ}}}

\newcommand{\WENO}{\ensuremath{\mathsf{WENO}}}
\newcommand{\WENOZ}{\ensuremath{\mathsf{WENOZ}}}

\newcommand{\WENOAO}{\ensuremath{\mathsf{WENO\mbox{-}AO}}}
\newcommand{\CWENOZAO}{\ensuremath{\mathsf{CWENOZ\mbox{-}AO}}}
\newcommand{\ENO}{\ensuremath{\mathsf{ENO}}}

\newcommand{\WAOBS}{\ensuremath{\mathsf{WAO\mbox{-}BGS}}}
\newcommand{\WAOAHZ}{\ensuremath{\mathsf{WAO\mbox{-}AHZ}}}
\newcommand{\CWZ}{\ensuremath{\mathsf{CWZ753}}}

\renewcommand{\vec}[1]{\mathbf{#1}}
\newcommand{\R}{\mathbb{R}}

\newcommand{\Ogrande}{\mathcal{O}}
\newcommand{\Opiccolo}{o}

\newcommand{\dx}{\mathrm{d}x}
\renewcommand{\d}{\mathrm{d}}
\newcommand{\ca}[1]{\overline{#1}}

\newcommand{\DX}{\mathrm{\Delta}x}

\newcommand{\Poly}[1]{\mathbb{P}^{#1}}
\newcommand{\Prec}{P_{\text{\sf rec}}}

\newcommand{\Popt}{P_{\text{\sf opt}}}
\newcommand{\Sopt}{\mathcal{S}_{\text{\sf opt}}}
\newcommand{\Sk}{\mathcal{S}_{k}}
\newcommand{\SQk}{\widetilde{\mathcal{S}}_{k}}

\newcommand{\pder}[2]{\frac{\partial #1}{\partial #2}}

\begin{document}

\title{Efficient implementation of adaptive order reconstructions%
 \thanks{This work was supported by the Deutsche Forschungsgemeinschaft (DFG, German Research Foundation) under Germany's Excellence Strategy -- EXC-2023 Internet of Production -- 390621612
 and 
 by INDAM GNCS-2019 grant ``Approssimazione numerica di problemi di natura iperbolica ed applicazioni''.
 }
}

\author{
	M. Semplice 
	\thanks{
		Dipartimento di Matematica -
		Universit\`a dell'Insubria --
		Via Valleggio, 11 - Como (Italy) --
		{\sl matteo.semplice@uninsubria.it}
	}
	\and
	G. Visconti
	\thanks{
		Institute f\"{u}r Geometrie und Praktische Mathematik -
		RWTH Aachen University --
		Templergraben 55, 52062 Aachen, Germany --
		{\sl visconti@igpm.rwth-aachen.de}
	}
}


\date{\today}

\maketitle

\begin{abstract}
Including polynomials with small degree and stencil when designing very high order  reconstructions is surely beneficial for their non oscillatory properties, but may bring loss of accuracy on smooth data unless special care is exerted.
In this paper we 
 address this issue with a new 
{Central \WENOZ\ (\CWENOZ)} approach, in which 
{the reconstruction polynomial is computed from a single set of non linear weights, but}
the linear weights of the polynomials with very low degree (compared to the final desired accuracy) are
{infinitesimal with respect to the grid size.} 
After proving general results that guide the choice of the \CWENOZ\ parameters, we study a concrete example of a reconstruction that blends polynomials of degree six, four and two,  mimicking  already published Adaptive Order \WENO\ reconstructions \cite{BGS:wao,AHZ18:wao}.
The {novel} reconstruction yields similar accuracy and oscillations with respect to the {previous} ones, but saves up to 20\% computational time since it does not rely on a hierarchic {approach} and thus does not compute multiple sets of nonlinear weights in each cell.


\subclass{
65D05 
\and 65M08 
\and 65M12 
\and 76M12 
}
\end{abstract}

\section{Introduction} \label{sec:AOintroduction}

This paper presents a novel, non hierarchic, construction that yields very high order essentially non oscillatory reconstructions that are useful in
high order numerical schemes for conservation laws.

Let us consider the balance law $\partial_t u + \partial_x f(u) = s(u)$ and discretize the domain in cells $\Omega_i$ for $i=1,\ldots,N$. Following the method of lines, we introduce the cell averages $\ca{u}_i(t) = \int_{\Omega_i} u(t,x) \mathrm{d}x$ and compute their approximations $\ca{U}_i(t)$ by numerically integrating the system of ordinary differential equations
\[
\begin{cases}
\frac{\mathrm{d}}{\mathrm{d}t}\ca{U}_i(t) = -\frac{1}{|\Omega_i|} [\mathcal{F}_{i+\nicefrac12}(t)-\mathcal{F}_{i-\nicefrac12}(t)]
+\mathcal{S}_i(t)
\\
\ca{U}_i(0)=\ca{u}_i(0)
\end{cases}
.
\]
Both the numerical fluxes $\mathcal{F}_{i\pm\nicefrac12}$ at the interfaces and the numerical source terms $\mathcal{S}_i$ are computed with the help of a reconstruction operator; this latter derives pointwise approximations $R_i(t,x)$, at location $x$ in the cell $i$, from the cell averages $\ca{U}_{i-l}(t),\ldots,\ca{U}_{i+r}(t)$, for some $l,r\geq0$. For a numerical flux function $F$ compatible with $f$, we define
$\mathcal{F}_{i+\nicefrac12}(t)=F\big(R_i(t,x_{i+\nicefrac12}) , R_{i+1}(t,x_{i+\nicefrac12})\big)$ and
the numerical source term is 
computed as
$\mathcal{S}_i(t) = \sum_{q=0}^{N_q} w_q s(R_i(t,x_q))$,
where $x_q$ and $w_q$ are the nodes and weights of a quadrature formula on $\Omega_i$.

The reconstruction operator should yield an accurate but non oscillatory pointwise approximation of the unknown, based on its cell averages.
Beyond the second order of accuracy, one considers essentially non oscillatory reconstructions, which are typically realized by selecting (as in \ENO\ \cite{HEOC:87}) or more commonly by blending in a nonlinear way polynomials with different degrees and/or stencils.
In particular \WENO, which was introduced in 
\cite{Shu97}, considers a set  of polynomials with equal degree but different stencils and aims at reproducing the accuracy of a higher degree interpolant when the data are locally smooth. The literature on this subject is vast and the reader may refer to \cite{Shu:97}
for a review.

The suboptimal accuracy close to critical points shown by the original design has been later overcome by new definitions of the nonlinear weights (e.g. mapped \WENO~\cite{HAP:2005:mappedWENO},  \WENOZ~\cite{BCCD:2008:wenoz5,CCD:11}, the global average weight of \cite{Baeza:19:CWENOglobalaverageweight}) or by taking the small $\epsilon$ parameter to be dependent on the local mesh size \cite{Arandiga:11}.

Another source of difficulty in \WENO\ reconstructions is the possible non-existence or non-positivity of the linear weights for certain grid {types} and reconstruction points \cite{ShiHuShu:2002}. 
A quite successful proposal 
{to address this issue} 
was put forward by Levy, Puppo and Russo in \cite{LPR:00:SIAMJSciComp} for the case of achieving a third order accurate reconstruction at cell center in one and two-dimensional uniform grids.
Their \CWENO3 reconstruction 
is a nonlinear blend of one second degree polynomial and of some first degree ones; this approach
frees the linear weights from having to satisfy accuracy requirements and allows to choose them arbitrarily, independently of the reconstruction point, independently of the grid type (Cartesian/unstructured, uniform/non-uniform, conforming/non-conforming, etc). Consequently no issues regarding their existence and positivity is present.
The accuracy of the \CWENO3 has been studied in
\cite{Kolb:14,CS:epsweno}.

The idea at the base of \CWENO3,
namely the use of the polynomial $P_0$ as in equation \eqref{eq:p0},
has been exploited in different setups.
Novel reconstructions of different orders of accuracy
appeared 
under various names
in the literature for the cases of one 
\cite{Capdeville:08,Baeza:19:CWENOglobalaverageweight},
two
\cite{Capdeville:11:cwenotriangular,SCR:CWENOquadtree,DBSR:ADER_CWENO,CWENOandaluz}
and three space dimensions
\cite{ZhouCai:08,LP:12,ZQ:CW3:tetrahedra,DBSR:ADER_CWENO}.
Among those, \cite{Capdeville:08,Capdeville:11:cwenotriangular,SCR:CWENOquadtree,ZQ:CW3:tetrahedra,DBSR:ADER_CWENO}
consider non-uniform grids.
Applications to stochastic Galerkin have been considered in \cite{GersterHerty} and to Hamilton-Jacobi equations in \cite{ZSQ:19:FDCWENO_HJ}.

One of the advantages of this approach is the possibility to achieve genuinely multi-dimensional reconstructions that do not rely on dimensional splitting and that are, theoretically and in practice, not more challenging than the one-dimensional counterpart (see e.g. \cite{SCR:CWENOquadtree} for AMR grids and \cite{ZQ:CW3:tetrahedra,DBSR:ADER_CWENO,DBSR:DG_CWENO} for simplicial ones).

\CWENO\ reconstructions of arbitrary high orders have been studied in \cite{CPSV:cweno} and the use of Z-weights has been pursued in \cite{CPSV:coolweno,CSV19:cwenoz}.%
A very important result of \cite{CSV19:cwenoz} is an analysis of the multidimensional oscillation indicators, leading to general results  that support the design of Z-weights at arbitrary order for very general one- and multi-dimensional finite volume grids.

The optimal convergence rate on smooth data can be easily achieved when the degree gap 
between the central high order and the lower degree 
polynomials is not too high
(see \cite{CPSV:cweno,CSV19:cwenoz}). 
Thus, in the design of very high order reconstructions, one has to employ ``low degree'' polynomials whose stencils are still quite large and that, as a consequence, are not very good at avoiding discontinuities in complex multi-dimensional flows.
For this reasons, many researchers have violated the classical hypothesis on the degree gap between high and low order polynomials when designing their reconstructions.
For example, 
small-stencil polynomials of degree one, irrespectively of the degree of the central polynomial 
in \cite{ZhuQiu:CW5P4P1,ZQ:CW4P3P1:triangular,DBSR:ADER_CWENO,CWENOandaluz}.

However, in order to include very low order polynomials in the pool of candidate reconstruction polynomials, special care must be exerted. 
A leap forward has been the proposal of the Adaptive Order $\WENOAO(r_{l},\ldots,r_{2};r_{1})$
by Balsara, Garain and Shu \cite{BGS:wao}.
For $l=2$ they essentially coincide with  $\CWENOZ$ reconstructions of \cite{CSV19:cwenoz}, and for $l>2$ the reconstruction is a blend of the reconstruction polynomials given by $\WENOAO(r_{k};r_1)$ for $k=2,\ldots,l$.
This hierarchic approach effectively enhances the stability of the reconstruction of order $r_{l}$,
but reduces the accuracy suddenly to $r_1$ whenever a discontinuity is present the stencils of both the polynomials of level $l$ and $l-1$.
Arbogast, Huang and Zhao in
\cite{AHZ18:wao} introduced a new hierarchical construction that instead is capable of reducing gradually the accuracy from $r_l$ to $r_1$ as the discontinuity moves inward in the reconstruction stencil.
$\WENOAO(7,5,3)$ and $\WENOAO(9,7,5,3)$ are considered in the numerical examples of the papers.

It should nevertheless be noted
that very high order Weighted Essentially Non-Oscillatory reconstructions are quite computationally intensive. For example, profiling the code 
{\tt claw1dArena}%
\footnote{
	For this test, {\tt claw1dArena} was compiled with the GNU Compiler and {\tt -O3} optimization level, profiling data were collected with the {\tt callgrind} utility of the {\tt valgrind} suite and analyzed with {\tt kcachegrind}. The data reported refer to the linear advection of the Jiang-Shu profile and to the Lax shock tube with characteristic projection.
}
\cite{claw1dArena} revealed that \CWENO, \CWENOZ\ and \WENOAO\ reconstructions of order 7 consume up to 80\% of the CPU time when 
the method of lines and
a simple numerical flux as local Lax-Friedrichs is employed. Turning on the local characteristic projection technique reduces the cost  of the reconstruction to around $50\div60\%$, which is still quite high.
This fully justifies the efforts in reducing the computational costs of very high order reconstructions. An example is given by the polynomial basis proposed in \cite{BGS:wao} that reduces the computational cost of the oscillation indicators.

In this paper we propose a novel approach to reduce the cost of adaptive order reconstructions.
In \S\ref{sec:AOpreliminary}, we revise the standard \CWENO\ \cite{CPSV:cweno} and \CWENOZ\ \cite{CSV19:cwenoz} reconstructions, as well as the \WENOAO\
reconstructions of \cite{BGS:wao,AHZ18:wao}.
In our approach, introduced and studied in 
\S\ref{sec:AOreconstruction}, we replace the hierarchical computation of nonlinear weights of \cite{BGS:wao,AHZ18:wao} with a hierarchy of scales in the linear weights
and only a set of nonlinear weights is computed.  
This hierarchy is realized by choosing, in a \CWENOZ\ approach, linear weights for the very low degree polynomials that are grid-size dependent. More precisely, for a reconstruction of accuracy $G+1$, all candidate polynomials of degree smaller than $G/2$ will have a linear weight of size $\Ogrande(\DX^{r})$ for some $r>0$.
The analysis in \S\ref{sec:AOreconstruction} indicates how the exponent $r$ should be chosen in order to guarantee both that, on the one hand, the accuracy of the reconstruction is not reduced in case of smooth data and, on the other hand, the non-oscillatory properties of the reconstruction are boosted by the very small stencils of some candidate polynomials close to discontinuities. 
In \S\ref{sec:AOsimulations} we report several numerical tests demonstrating that our new approach yields results that are comparable in accuracy to those obtained with the reconstructions in \cite{BGS:wao,AHZ18:wao}, but save up to $20\%$ of computational time.
Conclusions and perspectives for future work are discussed in \S\ref{sec:AOconclusion}.

\section{Review of $\CWENO$ based reconstructions} \label{sec:AOpreliminary}

Before introducing the new adaptive order $\CWENOZ$ reconstruction, we recall the definition of the classical $\CWENO$ reconstruction operator. Later in this section, $\CWENOZ$ and $\WENOAO$ will be introduced,
discussing their properties, but for a detailed analysis we point the reader to~\cite{CPSV:coolweno,CPSV:cweno,CS:epsweno} for $\CWENO$, to~\cite{CPSV:coolweno,CSV19:cwenoz} for $\CWENOZ$ and to~\cite{AHZ18:wao,BGS:wao,KC:18} for $\WENOAO$.

Throughout the paper we restrict the presentation to the scalar case since usually the reconstruction procedures are applied component-wise, directly to the conserved variables or after the local characteristic projection. We also concentrate on the one-dimensional setting, as multiple space dimensions may be treated either trivially by splitting or by extending the present hierarchic procedure to truly multi-dimensional \CWENOZ\ reconstructions along the lines of \cite{SCR:CWENOquadtree,DBSR:ADER_CWENO,ZQ:CW3:tetrahedra,ZQ:CW4P3P1:triangular,Capdeville:11:cwenotriangular}.

In order to describe a reconstruction procedure, we consider as given data the  cell averages $\ca{u}_k$ of a function ${u}$ over the cells of a grid that is composed by cells $\Omega_k\subset \R$. A reconstruction aims to recover point-wise information on $u$ in the interior and at the boundaries of a cell, using the knowledge of cell averages of $u$ in that cells and its neighbors.
To simplify the notation, we describe the reconstruction  in a cell $\Omega_0$ of size size $\Delta x$ and centered at the point $x_0= 0$.

\begin{definition} \label{def:conservation}
Let $\mathcal{S}$ be a set of cells containing $\Omega_0$.
The polynomial $P$ associated to the stencil $\mathcal{S}$ is the polynomial that satisfies the  conservation property
	\[ \tfrac{1}{\Delta x} \int_{\minor{\Omega_j}} P(x) \dx = 
	\minor{\ca{u}_j}, \quad \forall\,\minor{\Omega_j}\in\mathcal{S}.
	\]
\end{definition}

\minor{The interpolation property of all polynomials employed in this paper is intended in the sense of Definition~\ref{def:conservation}.} We point out that, especially in multi-dimensions, it is often convenient to relax the above requirement and employ least-squares fitted polynomials	as in \cite{HuShu:WENOtri,Capdeville:11:cwenotriangular,SCR:CWENOquadtree,DBSR:ADER_CWENO,ZQ:CW3:tetrahedra,ZQ:CW4P3P1:triangular}
and that the results of this paper, which rely only on the approximations errors of polynomials should easily extend to this more general setup.

Next we recall the Definition of the Jiang-Shu oscillation indicators.
Let $\Poly{k}$ be the space of polynomials with degree at most $k\in\mathbb{N}$.

\begin{definition}[see~\cite{JiangShu:96}] \label{def:ind}
	The smoothness indicator of a polynomial $P\in\Poly{k}$ is
	\begin{equation} \label{eq:ind}
	I[P] := 
	\sum_{i=1}^k  \Delta x^{2i-1} \int_{\Omega_0} \left(\frac{\d^i}{\dx^i} P(x)\right)^2 \dx. \end{equation}
\end{definition}

We recall that $I[P]=\Ogrande(1)$ even if a discontinuity is present in the stencil of $P$ and that $I[P]\to0$ under grid refinement if $P$ is associated to smooth data.

\revision{We point out that Definition~\ref{def:ind} for the smoothness indicator is not
	adequate for classes of problems where the solution is expected to be
	continuous and for example for Hamilton-Jacobi equations the summation
	in~\eqref{eq:ind} should start from $i=2$ in order to disregard the contribution of
	the $L^2$ norm of the first derivative~\cite{JiangPeng:00,FPT:AdaptiveFilteredHJ
	}. Also, different
	approaches to define smoothness indicators have been explored in the
	literature: see e.g.~\cite{HKLY:2013} for a definition based on $L^1$ norms. Of
	course, in order to employ different indicators in the constructions of
	this paper, one would need to first prove a counterpart of Proposition~\ref{th:proposition} of \S\ref{sssec:analysisSmooth}.}


\subsection{$\CWENO$ reconstructions} \label{ssec:cweno}

We recall here the general definition of a $\CWENO$ reconstruction, which was originally given in~\cite{LPR:00:SIAMJSciComp} for schemes of order $3$
and later
generalized to arbitrary order and analyzed in~\cite{CPSV:cweno}.

\begin{definition}[$\CWENO$ operator] \label{def:CWENO}
	Given a stencil $\Sopt$ \revision{of $G+1$ cells} that includes $\Omega_0$, let $\Popt\in\Poly{G}$ \emph{(optimal polynomial)} be the polynomial of degree $G$ which interpolates all the given data in $\Sopt$. Further, let $ {P}_1, {P}_2, \ldots, {P}_m$ be a set of $m\geq 1$ polynomials of degree $g$ with $g<G$ \revision{such that, for $k=1,\dots,m$, $P_k$ interpolates the cell averages of a sub-stencil $\mathcal{S}_k$, chosen such that} \revision{$\Omega_0\in\Sk\subset\Sopt$}.
	Let also $\{d_k\}_{k=0}^m$ be a set of strictly positive real coefficients such that $\sum_{k=0}^m d_k=1$.
	
	The $\CWENO$ operator computes a reconstruction polynomial
	\[ 
	\Prec 
	= \CWENO ({\Popt};{P}_1,\ldots,{P}_m) \in \Poly{G} 
	\]
	as follows:
	\begin{enumerate}
		\item  introduce the polynomial ${P}_0$ defined as
		\begin{equation} \label{eq:p0}
		{P}_0(x) = \frac{1}{d_0}\left({\Popt}(\vec{x})-\sum_{k=1}^{m}d_k {P}_k(x) \right) \in \Poly{G};
		\end{equation}
		\item compute the regularity indicators
		\[
		I_0=I[\Popt], 
		\qquad
		I_k=I[P_k], k\geq1
		\]
		\item compute  the \minor{nonlinear} coefficients $\{\omega_k\}_{k=0}^m$ as
		\begin{equation} \label{eq:Omega}
		\alpha_k = \frac {d_k} {\left( I_k+\epsilon \right)^{\ell}},
		\qquad
		\omega_k = \frac{\alpha_k}{\sum_{i=0}^{m}\alpha_i},
		\end{equation}
		where $\epsilon$ is a small positive quantity, and
		$\ell \ge 1$
		\item  define the reconstruction polynomial as
		\begin{equation}
		\Prec(x) 
		= \sum\limits_{k=0}^{m} \omega_k {P}_k(x) \in\Poly{G}. \label{eq:precCW}
		\end{equation}
	\end{enumerate}
\end{definition}

We point out that the use of the additional polynomial $P_0$ defined in equation~\eqref{eq:p0} is what characterizes a Central $\WENO$ based reconstruction. 
This
allows to employ linear coefficients $\{d_k\}_{k=0}^m$ which do not depend on the reconstruction point and consequently the $\CWENO$ operator defines a reconstruction polynomial which is globally defined and uniformly accurate on the computational cell. $\Prec$ can be later evaluated at any point, and no a-priori knowledge of reconstruction point is exploited in the computation of $\Prec$. This 
constitutes
also a computational gain since the \minor{nonlinear} coefficients can be computed once per cell and not once per reconstruction point, as in the standard $\WENO$~\cite{Shu:97}. This makes the $\CWENO$ 
procedures
more appealing for balance laws, multidimensional computations and unstructured meshes than 
their $\WENO$ counterparts. \revision{It should be noted that the reconstruction proposed in~\cite{DumbserKaeser:2007} also computes a single set of nonlinear weights per cell, but it requires much larger stencils than $\CWENO$.}

Next, 
we briefly present some results shared by all $\CWENO$ 
reconstructions.

\begin{remark}
	Since all the interpolating polynomials involved in Definition~\ref{def:CWENO} satisfy the conservation property on the reconstruction cell $\Omega_0$,
	$P_0$ is conservative \minor{on $\Omega_0$} in the sense of Definition~\ref{def:conservation} and thus also
	$\tfrac{1}{\Delta x} \int_{\Omega_0}\Prec(x) \dx 
	=\ca{u}_0
	$,
	which shows that the $\CWENO$ operator is conservative.
\end{remark}

The accuracy and non-oscillatory properties of $\CWENO$ based schemes are guaranteed by the dependence of their \minor{nonlinear} weights on the Jiang-Shu regularity indicators.  Since for any sufficiently regular function $u(x)$ the approximation orders of $\Popt \in \Poly{G}$ and $P_k \in \Poly{g}$, for $k=1,\ldots,m$, are
$$
\vert \Popt (x) - u(x) \vert 
= \Ogrande \big(\Delta x^{G+1}\big) 
\quad \text{and} \quad
\vert P_k (x) - u(x) \vert 
= \Ogrande \big(\Delta x^{g+1}\big) 
  \qquad
  \forall x\in\Omega_0,
$$ 
it is immediate to show that the reconstruction error at $x$ is at least of order $g+1$. Therefore, on smooth data, the \minor{nonlinear} weights must be designed so that the accuracy of the reconstruction polynomial $\Prec$ is boosted to the accuracy of $\Popt$, i.e. $G+1$. More precisely, if $d_k-\omega_k$ is at least $\Ogrande(\Delta x^{G-g})$ then the accuracy of the $\CWENO$ reconstruction equals the accuracy of $\Popt$. Provided that $G \leq 2g$ this condition is verified if $\epsilon = \Ogrande(\Delta x^{\hat{m}})$, $\hat{m}=1,2$, see~\cite[Proposition 1]{CPSV:cweno}.

On the other hand, if there were
an oscillating polynomial $P_{\hat{k}}$ for some $\hat{k}\in\{1,\dots,m\}$, then $I_{\hat{k}}\asymp1$\revision{, i.e.~$\lim_{\Delta x\to 0} I_{\hat{k}}$ exists, is finite and nonzero, and thus the} corresponding \minor{nonlinear} weight tends to $0$; moreover, as proved in~\cite{CPSV:cweno}, also the \minor{nonlinear} weight of $P_0$ would tend to $0$ and the reconstruction polynomial $\Prec$ would become a nonlinear combination of polynomials of degree $g$: the accuracy of the reconstruction reduces to $g+1$, but spurious oscillations in the PDE solution would be controlled.

The positive parameter $\epsilon$ prevents the division by zero in the computation of the \minor{nonlinear} weights~\eqref{eq:Omega}. In \cite{CS:epsweno,Kolb:14}, the authors proved that the choice of $\epsilon$ can influence the convergence of the method on smooth parts of the solution. For instance, choosing $\epsilon = \Ogrande(\Delta x^{\hat{m}})$, for some $\hat{m}\in\mathbb{N}$, helps to have a more
regular convergence history than taking $\epsilon$ as fixed value. A convergence analysis of the scheme on smooth data gives a range of values for $\hat{m}$ that guarantees optimal order; within this range one would take $\hat{m}$ as large as possible in order to avoid spurious oscillations on discontinuous data. In the following we always consider 
a mesh-size dependent $\epsilon$.

In Definition~\ref{def:CWENO} the number $m$ and the degree $g$ of the lower-degree polynomials is not specified nor linked to the degree $G$ of the optimal polynomial. However, the relation $G\leq 2g$ is required for accuracy and in a one space dimension reconstruction of order $2r-1$ it is customary to choose $g=r-1$, $m=g+1$ and $\Popt$ of degree $G=2r-2$. This latter is determined by the exact interpolation of the data in a symmetric stencil $\cal{S}_{\text{opt}}$ centered on $\Omega_0$ containing the cells $\Omega_{-g},\dots,\Omega_{g}$. Furthermore, for $k=1, \cdots, m$, the lower-degree polynomials ${P}_k$ are defined as the exact interpolants on the substencils
\revision{${\cal S}_k=\{\Omega_{k-r},\ldots,\Omega_{k-1} \} \subset {\cal S}_{\text{opt}}$}. This is the same choice considered in~\cite{CPSV:cweno,CPSV:coolweno}.

\subsubsection{$\CWENOZ$} \label{ssec:cwenoz}
$\CWENOZ$ was originally proposed in~\cite{CPSV:coolweno} and, recently, it was extensively analyzed in~\cite{CSV19:cwenoz} in a multi-dimensional setting. We recall here its definition.

\begin{definition}[$\CWENOZ$ reconstruction] \label{def:CWENOZ}
	Given a stencil $\Sopt$ \revision{of $G+1$ cells} that includes $\Omega_0$, let $\Popt\in\Poly{G}$ \emph{(optimal polynomial)} be the polynomial of degree $G$ which interpolates all the given data in $\Sopt$. Further, let $ {P}_1, {P}_2, \ldots, {P}_m$ be a set of $m\geq 1$ polynomials of degree $g$ with $g<G$ \revision{such that, for $k=1,\dots,m$, $P_k$ interpolates the cell averages of a sub-stencil $\mathcal{S}_k$, chosen such that} \revision{$\Omega_0\in\Sk\subset\Sopt$}.
	Let also $\{d_k\}_{k=0}^m$ be a set of strictly positive real coefficients such that $\sum_{k=0}^m d_k=1$.
	
	The $\CWENOZ$ operator 
	\[ 
		\Prec 
		= \CWENOZ ({\Popt};{P}_1,\ldots,{P}_m) \in \Poly{G} 
	\]
	is defined as the \CWENO\ operator of Definition~\ref{def:CWENO}, but replacing the definition of the \minor{nonlinear} coefficients \eqref{eq:Omega} with $\{\omega_k\}_{k=0}^m$defined as
	\begin{equation} \label{eq:OmegaZ}
		\alpha_k = {d_k} \left( 1 + \left( \frac {\tau} {I_k+\epsilon} \right)^{\ell} \right),
		\qquad
		\revision{\omega_k = \frac{\alpha_k}{\sum_{i=0}^{m}\alpha_i}},
	\end{equation}
	where $\epsilon$ is a small positive quantity, $\ell \ge 1$
	and $\tau$ is a global smoothness indicator. For efficiency, $\tau$ is restricted to be a linear combination of the other smoothness indicators $I_0=I[\Popt],\ldots,I_m=I[P_m]$, i.e. 
	\begin{equation}  \label{eq:tau}
		\tau := \left| \sum_{k=0}^m \lambda_k I_k \right|
	\end{equation}
for some choice of real coefficients $\lambda_0,\ldots,\lambda_m$.
\end{definition}

In~\cite{CSV19:cwenoz}, it is argued the coefficients should satisfy $\sum_{k=0}^{m}\lambda_k=0$ and it is also shown how to optimize their choice.

We point out that the $\CWENOZ$ reconstruction is a $\CWENO$ based reconstruction and therefore it shares the advantages of $\CWENO$ discussed in Section~\ref{ssec:cweno}. However, the $\CWENOZ$ differs from the classical definition of the $\CWENO$ scheme in the computation of the \minor{nonlinear} coefficients~\eqref{eq:OmegaZ}: the $\CWENOZ $ method uses the idea of Borges, Carmona, Costa and Don in~\cite{BCCD:2008:wenoz5}, where they introduced in a $\WENO$ setting definition \eqref{eq:OmegaZ} of the \minor{nonlinear} coefficients, which drives them closer to their optimal values in the smooth case.

The definition of the new \minor{nonlinear} weights guarantees
a weaker condition on $\hat{m}$ to be satisfied in order to reach optimal accuracy,  which implies the possibility to employ a smaller $\epsilon$ than in the corresponding \CWENO\ reconstruction, enhancing their non-oscillatory properties.  In~\cite{CSV19:cwenoz} a thorough analysis of the global smoothness indicator is performed and the optimal definition of the \minor{nonlinear} weights in multi-dimensional $\CWENOZ$ reconstructions is given. In particular, for conditions on $\ell$ and $\hat{m}$ we refer to~\cite[Theorem 24]{CSV19:cwenoz}.

\subsubsection{$\WENOAO$} \label{ssec:wenoAO}

Adaptive order reconstructions based on the $\CWENO$ procedure have been considered in recent years. We focus on the $\WENOAO$ schemes and the formulations given in~\cite{BGS:wao} and~\cite{AHZ18:wao}, 
recasting them as hierarchic $\CWENOZ$ 
reconstructions. In particular, 
we focus on
the $\WENOAO$ scheme of order $7$ as given in~\cite{BGS:wao} since it will be later used for comparisons with our $7$-th order adaptive reconstruction introduced in Section~\ref{ssec:examples}. 

Let $\Sopt$ be a stencil of seven cells centered on the computational cell $\Omega_0$. The $\WENOAO(7,5,3)$ reconstruction may be described as a \minor{nonlinear} blending of $\CWENOZ$ 
reconstructions
and, using the notation of this work, it can be thus formulated as
\begin{equation} \label{eq:wao753}
\begin{aligned}
	\Prec(x) &= 
	\CWENOZ\Big(
	    \Prec^{(7)};
	    \Prec^{(5)})
	  \Big), 
	  \quad\text{where}\\
	\Prec^{(7)} &= \CWENOZ(P^{(7)};P^{(3)}_L,P^{(3)}_C,P^{(3)}_R), \\
	\Prec^{(5)} &= \CWENOZ(P^{(5)};P^{(3)}_L,P^{(3)}_C,P^{(3)}_R)
\end{aligned}
\end{equation}
and
$P^{(7)}$ is the $6$-th degree polynomial interpolating the data in $\Sopt$, $P^{(5)}$ is the $4$-th degree polynomial interpolating the data 
in $\{\Omega_{-2},\ldots,\Omega_{2}\}$
and, finally, $P^{(3)}_L$, $P^{(3)}_C$, $P^{(3)}_R$ are the parabolas interpolating the data of the substencils
$\{\Omega_{k-3},\Omega_{k-2},\Omega_{k-1}\}$ for $k=1,2,3$, respectively.

As consequence of definition~\eqref{eq:wao753}, we have that the reconstruction polynomial $\Prec$ has the classical property of being globally defined on the computational cell $\Omega_0$ and being there uniformly accurate. However, it is clear that in order to compute $\Prec$ we need to proceed through three reconstruction steps. First, two inner $\CWENOZ$ procedures are applied in order to define two reconstruction polynomials of order $7$ and $5$ respectively.
Finally, $\Prec$ is computed with a $\CWENOZ$ operator having inputs $\Prec^{(7)}$ and $\Prec^{(5)}$. For a detailed description of the linear and \minor{nonlinear} weights employed in the $\WENOAO$ reconstruction we refer to~\cite{BGS:wao}. This procedure requires the computation of three sets \minor{nonlinear} weights and two intermediate polynomials, which
makes this approach computationally expensive, especially in view of generalizations to higher orders and thus deeper hierarchies.

A more general $\WENOAO$ reconstruction has been proposed in~\cite{AHZ18:wao} which has no base levels and, in addition, the definition has been slightly modified leading to a gain in the computational cost. In fact, the corresponding scheme of order $7$ requires the computation of only two sets of \minor{nonlinear} weights
instead of three.

\revision{Further, we recall that $\WENOAO$ reconstructions have been recently designed for unstructured meshes~\cite{BGFB:2020} and for $\mathrm{P_NP_M}$ schemes~\cite{BB:2019}.}

In the numerical tests of this paper we will refer to
the $\WENOAO(7,5,3)$ reconstruction of~\cite{AHZ18:wao}  with \WAOAHZ\ 
and with \WAOBS\ to the 
one
of~\cite{BGS:wao}, but in its finite volume formulation as given in~\cite{AHZ18:wao}.

\section{New adaptive order $\CWENOZ$ reconstruction} \label{sec:AOreconstruction}

Both $\WENOAO$ reconstructions \cite{BGS:wao,AHZ18:wao} share the following feature: more than one set of \minor{nonlinear} weights must be computed for each computational cell.
Our goal is
to introduce an adaptive order reconstruction in 
which 
multiple computation of \minor{nonlinear} weights are not needed, with the aim of saving computational time. 

\subsection{Definition and accuracy analysis} \label{ssec:cwenozAO}

The definition of the new method relies on the following generalization of $\CWENOZ$
and as such it belongs to the class of $\CWENO$
reconstructions.
The hierarchy of nested $\CWENOZ$ operators of \cite{BGS:wao,AHZ18:wao} is replaced by a hierarchy in the size of the linear weights, generalizing the approach of \cite{KolbSemplice:boundary}.

\begin{definition}[$\CWENOZAO$ reconstruction.] \label{def:CWENOZAO}
	Given a stencil $\Sopt$ \revision{of $G+1$ cells} that includes $\Omega_0$, let $\Popt\in\Poly{G}$ (\emph{optimal polynomial}) be the polynomial of degree $G$ associated to $\Sopt$. Let $ {P}_1, {P}_2, \ldots, {P}_m$ be a set of $m\geq 1$ polynomials  of degree $g$ with $\nicefrac{G}{2} \leq g<G$ \revision{such that, for $k=1,\dots,m$, $P_k$ is associated to a substencil $\Sk$ with}
	\revision{$\Omega_0\in\Sk\subset\Sopt$}.
	Further, let $ {Q}_1, {Q}_2, \ldots, {Q}_n$ be a set of $n\geq 1$ polynomials of degree $\gamma_k$ with $\gamma_k<\nicefrac{G}{2}$, for $k=1,\dots,n$, \revision{and such that, for $k=1,\dots,n$, $Q_k$ is associated to a substencil $\SQk$ with} \revision{$\Omega_0\in\SQk\subset\Sopt$}.
	
	The adaptive order $\CWENOZAO$ reconstruction computes a reconstruction polynomial by means of a $\CWENOZ$ operator
	\[  
	\begin{aligned}
	\Prec &= 
		\CWENOZAO (
		{\Popt}
		\,;\;
		{P}_1,\ldots,{P}_m
		\,;\;
		{Q}_1,\ldots,{Q}_n
		) 
		\\
		&=
		\CWENOZ (
		{\Popt}
		\,;\;
		{P}_1,\ldots,{P}_m,
		{Q}_1,\ldots,{Q}_n
		)
	\end{aligned}
	\]
	where the linear coefficients for the polynomials $Q_1,\ldots,Q_n$ are infinitesimal.\\
	In particular we consider $\delta_k=\DX^{r_k}$ for some $r_k>0$,
	for $k=1,\ldots,n$,
	and a set of strictly positive real coefficients $\{d_k\}_{k=0}^m$ 
	such that 
	$\sum_{k=0}^m d_k + \sum_{k=1}^n \delta_k=1$.
	The reconstruction is computed as follows:
	\begin{equation} \label{eq:precCWZAO}
		\Prec(x) 
		= \sum\limits_{k=0}^{m} \omega^P_k 
		P_k(x) + \sum\limits_{k=1}^{n} \omega^Q_k 
		Q_k(x) \in\Poly{G},
	\end{equation}
	where $\omega^P_k$ and $\omega^Q_k$ are the \minor{nonlinear} weights computed as in \minor{equation~\eqref{eq:OmegaZ}} 
	for the polynomial $P_k$ and 
	for the polynomial $Q_k$, respectively. \revision{We employ a different notation for the two subsets of nonlinear weights only for convenience in the proofs. However, the only difference between $\omega^P_k$ and $\omega^Q_k$ is the fact that $\omega^P_k$ results from~\eqref{eq:OmegaZ} with a linear weight $d_k = \Ogrande(1)$, while $\omega^Q_k$ results from~\eqref{eq:OmegaZ} but has an infinitesimal linear weight $\delta_k = \Opiccolo(1)$.}
	
	The additional polynomial ${P}_0$ is defined as
		\begin{equation} \label{eq:p0AO}
		{P}_0(x) = \frac{1}{d_0}\left({\Popt}(x)-\sum_{k=1}^{m}d_k {P}_k(x) -\sum_{k=1}^{n} \delta_k {Q}_k(x)\right) \in \Poly{G}.
		\end{equation}
		The global smoothness indicator $\tau$ is given as a linear combination of $I_0,\ldots,I_m^P$:
		\begin{equation} \label{eq:tauAO}
		\tau := \left| \sum_{k=0}^m \lambda_k I_k^P \right|
		\end{equation}
		for some choice of coefficients $\lambda_0,\dots,\lambda_m$ such that $\sum_{k=0}^m \lambda_k = 0$.
\end{definition}

We point out that the role of $\tau$ is to detect smoothness in the large stencil of $\Popt$. In view of this we did not find useful to include any of the indicators $I^Q_k$ in its definition.

Notice that in Definition~\ref{def:CWENOZAO} we already fix the relation between $G$ and $g$ as $G \geq 2g$. In addition, and compared to the $\CWENOZ$ reconstruction, the new procedure allows the use of polynomials $Q_k$ having degree $\gamma_k < g$. Classical schemes are characterized by the fact that for very large $G$, the stencil of the polynomials of degree $g$ are still quite large and it may be difficult to avoid discontinuities, especially in multi-dimension. Therefore, using lower degree polynomials helps to select smooth stencils but it may influence the optimal accuracy of the method on smooth data. However the fact that the linear weights $\delta_k$ associated to the low order polynomials $Q_k$ are $\delta_k = \Opiccolo(1)$, allows to reach accuracy on smooth data without relying on hierarchic \minor{nonlinear} blending of $\CWENO$ operators, which requires multiple computations of the \minor{nonlinear} weights as in the WENO-AO and similar reconstructions \cite{BGS:wao,AHZ18:wao}.

\subsubsection{Accuracy analysis on smooth solutions} \label{sssec:analysisSmooth}

We point out that the polynomial $P_0$ appearing in 
Definition~\ref{def:CWENOZAO} has degree $G$, but its accuracy may be degraded down to $\min_k\gamma_k<g$.
However, choosing $r_k\geq g-\gamma_k$, 
ensures that the accuracy of $P_0$ is unaffected by the presence of the polynomials $Q_1\ldots,Q_m$ in Definition~\ref{def:CWENOZAO}, as detailed in the following result.

\begin{lemma}\label{lem:rkg}
	Let $\Popt$, $P_k$, $k=1,\dots,m$, and $Q_k$, $k=1,\dots,n$, as in Definition~\ref{def:CWENOZAO}. Let $u$ be a sufficiently smooth function. The polynomial $P_0$ in equation~\eqref{eq:p0AO} is of degree $G$, but its accuracy order in the computational cell $\Omega_0$ is $g$, provided that $r_k \geq g-\gamma_k$.
\end{lemma}
\begin{proof}
	After the definition~\eqref{eq:p0AO} of $P_0$, for each $x$ in the computational cell $\Omega_0$ we have
	\begin{align*}
		P_0(x) - u(x) &= \frac{1}{d_0} \left[ \Popt(x) - \sum_{k=1}^m d_k P_k(x) - \sum_{k=1}^n \delta_k Q_k(x) - d_0 u(x)  \right] \\
		&= \frac{1}{d_0} \left[ \Popt(x) - \sum_{k=1}^m d_k P_k(x) - \sum_{k=1}^n \delta_k Q_k(x) - \left( 1 - \sum_{k=1}^m d_k - \sum_{k=1}^n \delta_k \right) u(x) \right] \\
		&= \frac{1}{d_0} \underbrace{\left(\Popt(x) - u(x)\right)}_{\Ogrande(\Delta x^{G+1})} + \frac{1}{d_0} \sum_{k=1}^n d_k \underbrace{\left(u(x) - P_k(x)\right)}_{\Ogrande(\Delta x^{g+1})}
		+ \frac{1}{d_0}
			\sum_{k=1}^m \delta_k \underbrace{\left(u(x) - Q_k(x)\right)}_{\Ogrande(\Delta x^{\gamma_k+1})}.
	\end{align*}
	Therefore, $P_0(x) - u(x) = \Ogrande(\Delta x^{g})$ if $\delta_k = \Ogrande(\Delta x^{r_k})$ with $r_k \geq g-\gamma_k$. \qed
\end{proof}

We can now state sufficient conditions to ensure that, for smooth data,  the reconstruction error at each point in the computational cell of the new $\CWENOZAO$ scheme is not bigger than the interpolation error of $\Popt$.

\begin{lemma} \label{lem:suffCond}
	Let $\Prec$ be a polynomial as in equation~\eqref{eq:precCWZAO} in Definition~\eqref{def:CWENOZAO}. Let be $u$ a sufficiently smooth function. Then, the approximation error of $\Prec$ in the computational cell $\Omega_0$ is of order $G+1$ under the restriction $\gamma_k+r_k\geq g$ provided that both $\frac{w_k^P-d_k}{d_k} = \Ogrande(\Delta x^{G-g})$ and $\frac{w_k^Q-\delta_k}{\delta_k} = \Ogrande(\Delta x^{G-\gamma_k-r_k})$ hold true.
\end{lemma}
\begin{proof}
The approximation errors of the polynomials appearing in Definition~\ref{def:CWENOZAO} are
$$
|\Popt(x) - u(x)| = \Ogrande(\Delta x^{G+1}), \quad |P_k(x) - u(x)| = \Ogrande(\Delta x^{g+1}), \quad |Q_k(x) - u(x)| = \Ogrande(\Delta x^{\gamma_k+1})
$$
at any point $x$ in the computational cell. Then the reconstruction error at $x$
\begin{align*}
u(x) - \Prec(x) = & \underbrace{(u(x)-\Popt(x))}_{O(\Delta x^{G+1})} + (d_0 - w^P_0) \underbrace{(P_0(x) - u(x))}_{O(\Delta x^{g+1})} \\
& + \sum_{k=1}^m (d_k-w_k^P) \underbrace{(P_k(x)-u(x))}_{O(\Delta x^{g+1})} + \sum_{k=1}^n (\delta_k-w_k^Q) \underbrace{(Q_k(x)-u(x))}_{O(\Delta x^{\gamma_k+1})}
\end{align*}
is of optimal order $G+1$ if $w^P_k-d_k=\Ogrande(\DX^{G-g})$ and $w^Q_k-\delta_k=\Ogrande(\DX^{G-\gamma_k})$. The first is equivalent to $\frac{w_k^P-d_k}{d_k} = O(\Delta x^{G-g})$ since $d_k=\Ogrande(1)$. The latter is implied by 
$\frac{w_k^Q-\delta_k}{\delta_k} = \Ogrande(\Delta x^{G-\gamma_k-r_k})$.
\end{proof}

The sufficient conditions of the previous result clarify the roles played by the nonlinear weights and by the linear weights in ensuring the optimal convergence rates on smooth data. The computation of the nonlinear weights should ensure that, for all polynomials, they are close to their linear counterpart in a relative sense, as $\Ogrande(\DX^{G-g})$, independently of $\gamma_k$. On the other hand, the linear weights $\delta_k=\DX^{r_k}$ make up for the difference $g-\gamma_k$.

For the analysis, let us assume that the cell averages to which the reconstruction is applied are sampled exactly from a sufficiently smooth function $u(x)$ that may have a critical point of order $n_{cp}\geq0$ at a point $x_0\in\Omega_0$. Obviously $n_{cp}=0$ means that there is no critical point.

Let us first prove a very general feature of the smoothness indicators.
\begin{proposition} \label{th:propI}
	Let $\mathcal{S}$ be a stencil including $\Omega_0$ and let $q(x)$ be a polynomial approximating a regular function $u(x)$. Then
	$$
	I[q] = \Ogrande(\DX^{2 n_{cp} + 2}).
	$$
\end{proposition}
\begin{proof}
	Since the regularity indicators are invariant if data are shifted by an additive constant, without loss of generality we assume that $u(x_0)=0$.
	In the following, we will consider the Taylor expansion of $u(x)$ around $x_0$ and denote for simplicity $u^{(k)}_0=u^{(k)}(x_0)$.
	The cell average of $u(x)$ in $\Omega_\alpha$ will have an expansion of the form
	$$
	\ca{u}_\alpha = \sum_{k \geq n_{cp}+1} A_k^\alpha u_0^{(k)} \DX^k
	$$
	for some constants $A_k^\alpha\in\R$.
	The regularity indicator is quadratic with respect to the data (see~\cite[Proposition 13]{CSV19:cwenoz}), and thus there exists $C_{\alpha\beta}\in\R$ such that
	\begin{align*}
	I[q] &= \revision{ \sum_{ \{ \alpha,\beta : \, \Omega_\alpha,\Omega_\beta \in \mathcal{S} \} } } C_{\alpha\beta} \ca{u}_\alpha \ca{u}_\beta \\
	&=  \revision{ \sum_{ \{ \alpha,\beta : \, \Omega_\alpha,\Omega_\beta \in \mathcal{S} \} } } C_{\alpha\beta} 
	\left( \sum_{k \geq n_{cp}+1} A_k^\alpha u_0^{(k)} \DX^k \right) 
	\left( \sum_{j \geq n_{cp}+1} A_j^\beta u_0^{(j)} \DX^j \right) \\
	&= \Ogrande(\DX^{2 n_{cp} + 2}).
	\end{align*}
\end{proof}

Next, we recall some results proven in 
~\cite[Proposition 14, Corollary 16 and Corollary 22]{CSV19:cwenoz}.
\begin{proposition} \label{th:proposition}
	Let $\mathcal{S}$ be a stencil including $\Omega_0$ and let $q(x)$ be a polynomial approximating a regular function $u(x)$ with accuracy greater or equal than $M$, then
	$$
		I[q] = B_M + R[q]
	$$
	where $B_M$ depends on $M$ but not on $q(x)$, while $R[q]$ depends on the stencil $\mathcal{S}$ and $R[q]=\Opiccolo(B_M)$. Moreover, if $u$ has a critical point with $n_{cp} \geq M$, then $B_M = 0$ and $R[q] = \Ogrande(\Delta x^{2M+2})$. In the case $n_{cp} < M$, then $B_M = \Ogrande(\Delta x^{2(n_{cp}+1)})$ and $R[q] = \Ogrande(\Delta x^{M+2+n_{cp}})$, so that $R[q]=\Opiccolo(B_M)$. Further, thanks to the hypothesis $\sum_{k=0}^m \lambda_k = 0$, $\tau = \Ogrande(\sum_{k=0}^m R_k)$.
\end{proposition}

\begin{lemma}[Lemma 6 of \cite{DB:2013}] \label{lem:lemma}
	If $\alpha_k^P = d_k\left( 1+A\DX^\beta+\Ogrande(\DX^{\beta+1}) \right)$ for $k=0,\dots,m$, and $\alpha_k^Q = \delta_k\left( 1+A\DX^\beta+\Ogrande(\DX^{\beta+1}) \right)$ for $k=1,\dots,n$ with $\beta > 0$ and $A$ independent on $k$, then $\frac{\omega_k^P}{d_k} = 1 + \Ogrande(\DX^{\beta+1})$ for $k=0,\dots,m$ and $\frac{\omega_k^Q}{\delta_k} = 1 + \Ogrande(\DX^{\beta+1})$ for $k=1,\dots,n$.
\end{lemma}
\begin{proof}
	The proof relies on observing that
	$$
	\sum_{j=0}^m \alpha_j^P + \sum_{j=1}^m \alpha_j^Q = 1 + A \DX^\beta + \Ogrande(\DX^{\beta+1}).
	$$
	Then
	$$
	\omega_k^P = d_k \left( 1 + A \DX^\beta + \Ogrande(\DX^{\beta+1}) \right) \left( 1 - A \DX^\beta + \Ogrande(\DX^{\beta+1}) \right)
	$$
	and similarly for $\omega_k^Q$.
\end{proof}

We are now ready to analyze the behavior of the reconstruction on smooth data. In the following, given a function $q(\DX)$, we will write $\theta(q(\DX)) = r$ to mean that $q(\DX) \sim a_r \DX^r$ for some $a_r\neq0$.

\paragraph{Case (A): \pmb{$n_{cp}\geq g$}}
Applying Proposition~\ref{th:proposition} with $M=g$ we have that $B_g=0$ and thus, defining $R_k=R[P_k]$, we can write
for $k=0,\dots,m$
that
$$
\alpha_k^P = d_k\big(1+(a_k)^\ell\big)
\quad\text{ with }\quad
	a_k = \frac{\tau}{I_k^P + \epsilon} = \frac{\tau}{R_k+\epsilon} \sim \frac{C_\tau \DX^{\theta(\tau)}}{C_k \DX^{\theta(R_k)} + C_\epsilon \DX^{\hat{m}}}.
$$
For Proposition~\ref{th:proposition} we have that both $\theta(\tau)$ and $\theta(R_k)$ are larger than $2g+2$. Then
restricting the choice of $\hat{m}$ to
\begin{equation} \label{eq:mCondition}
\hat{m} \leq 2g+1
\end{equation}
ensures that
$a_k \sim C \Delta x^{\theta(\tau)-\hat{m}} \to 0$ with $C=C_\tau/C_\epsilon$ independent on $k$.
Similarly, for $k=1,\dots,n$, 
$$
\alpha_k^Q = \delta_k\big(1+(b_k)^\ell\big)
\quad\text{ with }\quad
	b_k = \frac{\tau}{I_k^Q + \epsilon}
	\sim \frac{C_\tau \DX^{\theta(\tau)}}{\tilde{C}_k \DX^{\theta(I_k^Q)} + C_\epsilon \DX^{\hat{m}}}
$$
and, using condition~\eqref{eq:mCondition},
$b_k \sim C \DX^{\theta(\tau)-\hat{m}} \to 0$ with $C=C_\tau/C_\epsilon$. This is trivial if $\gamma_k=0$ and otherwise it holds thanks to Proposition~\ref{th:propI}, which ensures $\theta(I^Q_k)\geq2g+2$.

We thus have that
$\alpha_k^P = d_k\big(1+C^\ell\DX^{\ell t}+\Ogrande(\DX^{\ell t+1}) \big)$
for $k=0,\ldots,m$ and 
$\alpha_k^Q = \delta_k\big(1+C^\ell\DX^{\ell t}+\Ogrande(\DX^{\ell t+1}) \big)$
for $k=1,\ldots,n$,
where we have defined $t=\theta(\tau)-\hat{m}$.
We can thus apply Lemma~\ref{lem:lemma} to conclude that 
$\frac{\omega_k^{P} - d_k}{d_k} =\Ogrande(\DX^{\ell t+1})$
and
$	\frac{\omega_k^Q - \delta_k}{\delta_k} =  \Ogrande(\DX^{\ell t+1}).
$
Finally,
the sufficient conditions of Lemma~\ref{lem:suffCond} for optimal convergence order are satisfied if
$$
	\ell(\theta(\tau)-\hat{m})+1 \geq G - g \quad \mbox{ and } \quad \ell(\theta(\tau)-\hat{m})+1 \geq G - \gamma_k - r_k, \ k=1,\dots,n
$$
or equivalently
\begin{equation} \label{eq:conditionA}
	\ell(\theta(\tau)-\hat{m}) \geq \max\{ G-g-1, \max_{k=1,\dots,n}\{ G-\gamma_k-r_k-1 \} \}.
\end{equation}

\paragraph{Case (B): \pmb{$n_{cp}<g$}}
Proposition~\ref{th:proposition} applied to $P_0(x),\dots,P_m(x)$ with $M=g$ ensures that $R_k = \Opiccolo(B_g)$ and thus, for $k=0,\dots,m$, 
$$
	\alpha_k^P = d_k \big(1 + (a_k)^\ell \big)
	\quad\text{ with }\quad
	a_k = \frac{\tau}{B_g + \epsilon} \frac{1}{1 + \frac{R_k}{B_g + \epsilon}}
	\sim
	\frac{\tau}{B_g + \epsilon}
	\sim
	\frac{C_\tau\DX^{\theta(\tau)}}
	     {C_g\DX^{2n_{cp}+2} + C_\epsilon\DX^{\hat{m}}}
	.
$$
Similarly, for $k=1,\dots,n$, 
we can use Proposition~\ref{th:propI} to obtain
$$
	\alpha_k^Q = \delta_k\big(1+(b_k)^\ell\big)
	\quad\text{ with }\quad
	b_k = \frac{\tau}{I_k^Q + \epsilon}
	\sim
	\frac{C_\tau\DX^{\theta(\tau)}}
	{\widehat{C}_k\DX^{2n_{cp}+2} + C_\epsilon\DX^{\hat{m}}}.
$$
We observe that if $2 n_{cp} + 2 > \hat{m}$, i.e. $\frac{\hat{m}-2}{2} < n_{cp} < g$, then $a_k,b_k \sim C \Delta x^{\theta(\tau)-\hat{m}} \to 0$ with $C=C_\tau/C_\epsilon$ independent on $k$. Consequently, we are in the same situation as Case (A) and we can apply Lemma~\ref{lem:lemma} and Lemma~\ref{lem:suffCond} to prove the sufficient conditions for optimal convergence given in~\eqref{eq:conditionA}.

In particular, $\hat{m}=1$ is always included in the analysis above. Instead, we must distinguish some sub-cases when $0 \leq n_{cp} \leq \frac{\hat{m}-2}{2}$, $\hat{m}>1$.

\paragraph{Case (B1): \pmb{$n_{cp}=\frac{\hat{m}-2}{2}, \ \hat{m} \geq 2$}} In this case, defining $t=\theta(\tau)-\hat{m}$, we have
$$
	a_k \sim \frac{C_\tau}{C_g + C_\epsilon}\DX^t
	\quad \mbox{ and } \quad
	b_k \sim \frac{C_\tau}{\widehat{C}_k + C_\epsilon}\DX^t.
$$
We cannot directly apply Lemma~\ref{lem:lemma}, but, defining $C = \frac{C_\tau}{C_g + C_\epsilon}$ and $\widetilde{C}_k = \frac{C_\tau}{\widehat{C}_k + C_\epsilon}$, we compute
\begin{align*}
	\sum_{k=0}^m \alpha_k^P + \sum_{k=1}^n \alpha_k^Q 
	&= \sum_{k=0}^m d_k \left( 1 + C^\ell \DX^{\ell t} \right)
	+ \sum_{k=1}^{n} \delta_k \left( 1 + \widetilde{C}_k^\ell \DX^{\ell t} \right) \\
	&= 1 
	+ \left( 1 - \sum_{k=1}^n \delta_k \right) C^\ell \DX^{\ell t} 
	+ \sum_{k=1}^{n} \widetilde{C}_k^\ell \DX^{\ell t + r_k} \\	
	&\sim 1 + C^\ell \DX^{\ell t}.
\end{align*}
Then, using~\eqref{eq:OmegaZ}, for $k=0,\dots,m$ we compute
$$
	\omega_k^P \sim  d_k \left( 1 + C^\ell \DX^{\ell t} \right) \left( 1 - C^\ell \DX^{\ell t} \right) = d_k \left( 1 + \Ogrande(\DX^{\ell t+1}) \right)
$$
and similarly for $k=1,\dots,n$ we have
$$
	\omega_k^Q \sim \delta_k \left( 1 + \widetilde{C}_k^\ell \DX^{\ell t} \right) \left( 1 - C \DX^{\ell t} \right) = \delta_k (1+\Ogrande(\DX^{\ell t})).
$$
Thus, sufficient conditions of Lemma~\ref{lem:suffCond} for optimal convergence are satisfied if
$$
	\ell(\theta(\tau)-\hat{m})+1 \geq G - g \quad \mbox{ and } \quad \ell(\theta(\tau)-\hat{m}) \geq G - \gamma_k - r_k, \ k=1,\dots,n
$$
or equivalently
\begin{equation}
\label{eq:conditionB1}
	\ell(\theta(\tau)-\hat{m}) \geq \max\{ G-g-1, \max_{k=1,\dots,n}\{ G-\gamma_k-r_k \} \}.
\end{equation}

\paragraph{Case (B2): \pmb{$n_{cp}<\frac{\hat{m}-2}{2}, \ \hat{m}\geq 3$}}

The analysis of the previous cases was independent on the choice of the degrees $\gamma_k$, $k=1,\dots,n$. Now, instead, we need to consider explicitly the case when 
a constant polynomial is present, say $\gamma_1,\dots,\gamma_{n-1}>0$, while $\gamma_n=0$.

In this case, we have that
$$
	a_k \sim \frac{C_\tau}{C_g + C_\epsilon}\DX^t, \ \quad
	b_k \sim \frac{C_\tau}{\widehat{C}_k + C_\epsilon}\DX^t
	\quad \mbox{ and } \quad
	b_n \sim \frac{C_\tau}{C_\epsilon}\DX^{\widehat{t}}
$$
where we have defined $t=\theta(\tau)-2n_{cp}-2$ and $\widehat{t}=\theta(\tau)-\hat{m}$. Notice that $t > \widehat{t}$.

Defining $C = \frac{C_\tau}{C_g + C_\epsilon}$, $\widetilde{C}_k = \frac{C_\tau}{\widehat{C}_k + C_\epsilon}$ and $\widetilde{C}_n = \frac{C_\tau}{C_\epsilon}$, we have that
\begin{align*}
	\sum_{k=0}^m \alpha_k^P + \sum_{k=1}^n \alpha_k^Q 
	&= \sum_{k=0}^m d_k \left( 1 + C^\ell \DX^{\ell t} \right)
	+ \sum_{k=1}^{n-1} \delta_k \left( 1 + \widetilde{C}_k^\ell \DX^{\ell t} \right) 
	+ \delta_n \left( 1 + \widetilde{C}_n^\ell \DX^{\ell \widehat{t}} \right) \\
	&= 1 
	+ \left( 1 - \sum_{k=1}^n \delta_k \right) C^\ell \DX^{\ell t} 
	+ \sum_{k=1}^{n-1} \tilde{C}_k^\ell \DX^{\ell t + r_k}
	+ \left( 1 + \tilde{C}_n^\ell \DX^{\ell \widehat{t} + r_n} \right) \\	
	&\sim 1 + C^\ell \DX^{\ell t}
\end{align*}
only if $r_n > \ell(t-\widehat{t}) = \ell(\hat{m}-2n_{cp}-2)$.

Then, using~\eqref{eq:OmegaZ}, for $k=0,\dots,m$ we compute
$$
\omega_k^P \sim  d_k \left( 1 + C^\ell \DX^{\ell t} \right) \left( 1 - C^\ell \DX^{\ell t} \right) = d_k \left( 1 + \Ogrande(\DX^{\ell t+1}) \right)
$$
and similarly for $k=1,\dots,n$ we have
$$
\omega_k^Q \sim \delta_k \left( 1 + \widetilde{C}_k^\ell \DX^{\ell t} \right) \left( 1 - C \DX^{\ell t} \right) = \delta_k (1+\Ogrande(\DX^{\ell t}))
$$
for $k=1,\ldots,n-1$ and
$$
\omega_n^Q \sim \delta_k \left( 1 + \widetilde{C}_n^\ell \DX^{\ell\widehat{t}} \right) \left( 1 - C \DX^{\ell t} \right) = \delta_n (1+\Ogrande(\DX^{\ell\widehat{t}})).
$$
Thus, sufficient conditions of Lemma~\ref{lem:suffCond} for optimal convergence are satisfied if
$$
\ell(\theta(\tau)-2n_{cp}-2)+1 \geq G - g, \ \quad \ell(\theta(\tau)-2n_{cp}-2) \geq G - \gamma_k - r_k, \ k=1,\dots,n-1
$$
and
$$
	\ell(\theta(\tau)-\hat{m}) \geq G - r_n.
$$
In the general case these conditions become
\begin{subequations} \label{eq:conditionB2}
	\begin{gather}
	\ell(\theta(\tau)-2n_{cp}-2) \geq \max\{ G-g-1, \max_{\{ k : \, \gamma_k > 0 \}}\{ G-\gamma_k-r_k \} \} \label{eq:conditionB2nonconst}
	\\
	\ell(\theta(\tau)-\hat{m}) \geq G - r_k, \quad \{ k: \, \gamma_k = 0 \}. \label{eq:conditionB2const}
	\end{gather}
\end{subequations}

\renewcommand{\arraystretch}{1.25}
\begin{table}[t!]
	\centering
	\caption{Summary of the sufficient conditions, and related cases, for optimal convergence depending on the values of $n_{cp} \leq 4$ and $\hat{m} \leq 4 $.\label{tab:generalCondit}}
	\begin{tabular}{ll|c|c|c|c|c|}
		\cline{3-7}
		&   & \multicolumn{5}{c|}{$n_{cp}$}                                                                                                                                                                                                                                                                                                                                                                                                                                                  \\ \cline{3-7}
		&   & 0                                                                                          & 1                                                                                          & 2                                                                                          & 3                                                                                          & 4                                                                                          \\ \hline
		\multicolumn{1}{|l|}{\multirow{7}{*}{$\hat{m}$}} & 1 & \begin{tabular}[c]{@{}c@{}}Eq.~\eqref{eq:conditionA} \\ \textcolor{gray}{(A)-(B)}\end{tabular} & \begin{tabular}[c]{@{}c@{}} Eq.~\eqref{eq:conditionA} \\ \textcolor{gray}{(A)-(B)} \end{tabular} & \begin{tabular}[c]{@{}c@{}} Eq.~\eqref{eq:conditionA} \\ \textcolor{gray}{(A)-(B)} \end{tabular} & \begin{tabular}[c]{@{}c@{}} Eq.~\eqref{eq:conditionA} \\ \textcolor{gray}{(A)-(B)} \end{tabular} & \begin{tabular}[c]{@{}c@{}} Eq.~\eqref{eq:conditionA} \\ \textcolor{gray}{(A)-(B)} \\\end{tabular} \\ \cline{2-7} 
		\multicolumn{1}{|l|}{}                           & 2 & \begin{tabular}[c]{@{}c@{}} Eq.~\eqref{eq:conditionB1} \\ \textcolor{gray}{(B1)}\end{tabular}     & \begin{tabular}[c]{@{}c@{}} Eq.~\eqref{eq:conditionA} \\ \textcolor{gray}{(A)-(B)} \end{tabular} & \begin{tabular}[c]{@{}c@{}} Eq.~\eqref{eq:conditionA} \\ \textcolor{gray}{(A)-(B)} \end{tabular} & \begin{tabular}[c]{@{}c@{}} Eq.~\eqref{eq:conditionA} \\ \textcolor{gray}{(A)-(B)}\end{tabular} & \begin{tabular}[c]{@{}c@{}} Eq.~\eqref{eq:conditionA} \\ \textcolor{gray}{(A)-(B)}\end{tabular} \\ \cline{2-7} 
		\multicolumn{1}{|l|}{}                           & 3 & \begin{tabular}[c]{@{}c@{}} Eq.~\eqref{eq:conditionB2} \\ \textcolor{gray}{(B2)}\end{tabular}     & \begin{tabular}[c]{@{}c@{}} Eq.~\eqref{eq:conditionA} \\ \textcolor{gray}{(A)-(B)}\end{tabular} & \begin{tabular}[c]{@{}c@{}} Eq.~\eqref{eq:conditionA} \\ \textcolor{gray}{(A)-(B)}\end{tabular} & \begin{tabular}[c]{@{}c@{}}Eq.~\eqref{eq:conditionA} \\ \textcolor{gray}{(A)-(B)}\end{tabular} & \begin{tabular}[c]{@{}c@{}}Eq.~\eqref{eq:conditionA} \\ \textcolor{gray}{(A)-(B)}\end{tabular} \\ \cline{2-7} 
		\multicolumn{1}{|l|}{}                           & 4 & \begin{tabular}[c]{@{}c@{}}Eq.~\eqref{eq:conditionB2} \\ \textcolor{gray}{(B2)}\end{tabular}     & \begin{tabular}[c]{@{}c@{}}Eq.~\eqref{eq:conditionB1} \\ \textcolor{gray}{(B1)}\end{tabular}    & \begin{tabular}[c]{@{}c@{}}Eq.~\eqref{eq:conditionA} \\ \textcolor{gray}{(A)-(B)}\end{tabular} & \begin{tabular}[c]{@{}c@{}}Eq.~\eqref{eq:conditionA} \\ \textcolor{gray}{(A)-(B)}\end{tabular} & \begin{tabular}[c]{@{}c@{}}Eq.~\eqref{eq:conditionA} \\ \textcolor{gray}{(A)-(B)}\end{tabular} \\ \hline
	\end{tabular}
\end{table}

Summarizing, relations~\eqref{eq:mCondition}, \eqref{eq:conditionA}, \eqref{eq:conditionB1} and~\eqref{eq:conditionB2} define the complete set of sufficient conditions for optimal convergence. In particular, once~\eqref{eq:mCondition} is applied in order to find an upper bound for $\hat{m}$, conditions on $\ell$ and $r_k$ can be found. We summarize them in Table~\ref{tab:generalCondit} and for some prototype reconstructions in the following section.

\subsection{A $7$-th order $\CWENOZAO$ reconstruction} \label{ssec:examples}

The definition of the $\CWENOZAO$ reconstruction, see Definition~\ref{def:CWENOZAO}, allows to define a very wide set of reconstructions characterized by a very high order gap between the optimal polynomial and lowest degree polynomials. Moreover, we stress the fact that $\CWENOZAO$ does not require a base level reconstruction, therefore sharing the feature of the $\WENOAO$ reconstruction by~\cite{AHZ18:wao}.

Here, we propose one adaptive order reconstruction of order $7$ which we name $\CWZ$ and it will be later numerically compared with the $\WENOAO$ reconstructions of order $7$ introduced in~\cite{AHZ18:wao,BGS:wao}. The $\CWZ$ is characterized by a stencil $\Sopt$ of seven cells centered on the computational cell $\Omega_0$ \revision{which is the same stencil on which the $\WENOAO$ reconstructions of~\cite{AHZ18:wao} and~\cite{BGS:wao} are built.} The optimal polynomial $\Popt \in \Poly{6}$ \minor{interpolates} the data in $\Sopt$ and \minor{the} polynomial $P_1 \in \Poly{4}$ \minor{interpolates} the data of the five \minor{central} cells \minor{of} $\Sopt$. 
The lower degree polynomials \minor{$Q_1, Q_2, Q_3 \in \Poly{2}$ are the parabolas} interpolating the data of the three substencils $\{\Omega_{k-3},\Omega_{k-2},\Omega_{k-1}\}$ for $k=1,2,3$.
Consequently, in the notation of Definition~\ref{def:CWENOZAO}, we are in the case of $G=6$, $g=4$ and $\gamma_k = 2$, for $k=1,2,3$, and the reconstruction polynomial is defined by
\begin{equation} \label{eq:cwz753}
	\Prec = \CWENOZAO ({\Popt};{P_1};{Q}_1,{Q}_2,{Q}_3) \in \Poly{6}.
\end{equation}
We employ as global smoothness indicator \minor{$\tau= \left|I_0 - I[P_1]\right|.$} Let $u=u(x)$ be a function continuously differentiable with $x\in\R$, and let $\Omega_0 = \left[ -\frac{\Delta x}{2} , \frac{\Delta x}{2} \right]$, then the Taylor expansion of the global smoothness indicator $\tau$ is 
$$
\begin{aligned}
	\tau = 
	&\tfrac{1}{15} u^\prime(0) u^{(5)}(0) \Delta x^6 
	\\
	&+ \left[ \tfrac{1}{40} u^\prime(0) u^{(7)}(0) 
	+\tfrac{31}{1260} u^{\prime\prime}(0) u^{(6)}(0) 
	-\tfrac{16403}{30240} u^{\prime\prime\prime}(0) u^{(5)}(0) \right] \Delta x^8 
	\\&+ \Ogrande(\Delta x^{10}).
\end{aligned}
$$
We use it to compute the parameters $\hat{m}$, $\ell$ and $r_k$, $k=1,2,3$, satisfying the sufficient conditions for optimal convergence. All the values, for each $n_{cp}$ and $\hat{m}$, are summarized in Table~\ref{tab:conditions7} by using the smaller value allowed for $g$, i.e. $g = \frac{G}{2} = 3$. Condition~\eqref{eq:mCondition} provides the restriction $\hat{m} \leq 7$.

\begin{remark}
	Both in~\cite{BGS:wao} and~\cite{AHZ18:wao}, a $\WENOAO$ reconstruction of order $9$ is considered, where the base level is again represented by the three parabolas as in the $\WENOAO(7,5,3)$ reconstruction. The ninth-order scheme computes a reconstruction polynomial as
	\begin{equation} \label{eq:wao9753}
		\Prec(x) = \CWENOZ\Big(\CWENOZ(P^{(9)},P^{(3)}_L,P^{(3)}_C,P^{(3)}_R);\WENOAO(7,5,3)\Big).
	\end{equation}
	The corresponding reconstruction based on the $\CWENOZAO$ definition would compute the reconstruction polynomial as
	\begin{equation} \label{eq:cwz9753}
		\Prec(x) = \CWENOZAO\Big(P^{(9)};P^{(7)},P^{(5)};P^{(3)}_L,P^{(3)}_C,P^{(3)}_R\Big).
	\end{equation}
	Comparing~\eqref{eq:wao9753} and~\eqref{eq:cwz9753}, we observe that $\CWENOZAO$ always computes one set of \minor{nonlinear} coefficients. The $\WENOAO$ scheme of~\cite{BGS:wao} needs five sets of \minor{nonlinear} weights, while the $\WENOAO$ scheme of~\cite{AHZ18:wao} requires three sets of \minor{nonlinear} coefficients.
	It is then clear that, increasing the order, the gap between the computational cost required by the $\CWENOZAO$ and the $\WENOAO$ reconstructions becomes wider. 
\end{remark}

%
%
%

\renewcommand{\arraystretch}{1.25}
\begin{table}[t!]
	\centering
	\caption{Sufficient conditions for optimal convergence for the reconstruction $\CWZ$ defined in Section~\ref{ssec:examples} and by equation~\eqref{eq:cwz753}.\label{tab:conditions7}}
	\begin{tabular}{ll|c|c|c|c||c|}
		\cline{3-7}
		&   & \multicolumn{4}{c|}{$n_{cp}$}                                                                                                                                                                                                                                                                                                                                   & \multirow{3}{*}{Summary}                                                            \\ \cline{3-6}
		&   & \begin{tabular}[c]{@{}c@{}}0\\ $\theta(\tau) = 6$\end{tabular} & \begin{tabular}[c]{@{}c@{}}1\\ $\theta(\tau) = 8$\end{tabular} & \begin{tabular}[c]{@{}c@{}}2\\ $\theta(\tau)=8$\end{tabular}                        & \begin{tabular}[c]{@{}c@{}}$\geq$3\\ $\theta(\tau) = 10$\end{tabular} &                                                                                     \\ \hline
		\multicolumn{1}{|l|}{\multirow{9}{*}{$\hat{m}$}} & 1 & $\ell\geq1$, $r_k\geq1$                                           & $\ell\geq1$, $r_k\geq1$                                           & $\ell\geq1$, $r_k\geq1$                                                                & $\ell\geq1$, $r_k\geq1$                                            & $\ell\geq1$, $r_k\geq1$                                                                \\ \cline{2-7} 
		\multicolumn{1}{|l|}{}                           & 2 & $\ell\geq1$, $r_k\geq1$                                           & $\ell\geq1$, $r_k\geq1$                                           & $\ell\geq1$, $r_k\geq1$                                                                & $\ell\geq1$, $r_k\geq1$                                            & $\ell\geq1$, $r_k\geq1$                                                                \\ \cline{2-7} 
		\multicolumn{1}{|l|}{}                           & 3 & $\ell\geq1$, $r_k\geq1$                                           & $\ell\geq1$, $r_k\geq1$                                           & $\ell\geq1$, $r_k\geq1$                                                                & $\ell\geq1$, $r_k\geq1$                                            & $\ell\geq1$, $r_k\geq1$                                                                \\ \cline{2-7} 
		\multicolumn{1}{|l|}{}                           & 4 & $\ell\geq1$, $r_k\geq1$                                           & $\ell\geq1$, $r_k\geq1$                                           & $\ell\geq1$, $r_k\geq1$                                                                & $\ell\geq1$, $r_k\geq1$                                            & $\ell\geq1$, $r_k\geq1$                                                                \\ \cline{2-7} 
		\multicolumn{1}{|l|}{}                           & 5 & $\ell\geq1$, $r_k\geq1$                                           & $\ell\geq1$, $r_k\geq1$                                           & $\ell\geq1$, $r_k\geq1$                                                                & $\ell\geq1$, $r_k\geq1$                                            & $\ell\geq1$, $r_k\geq1$                                                                \\ \cline{2-7} 
		\multicolumn{1}{|l|}{}                           & 6 & $\ell\geq1$, $r_k\geq1$                                           & $\ell\geq1$, $r_k\geq1$                                           & \begin{tabular}[c]{@{}c@{}}$\ell\geq2$, $r_k\geq1$\\ $\ell\geq1$, $r_k\geq2$\end{tabular} & $\ell\geq1$, $r_k\geq1$                                            & \begin{tabular}[c]{@{}c@{}}$\ell\geq2$, $r_k\geq1$\\ $\ell\geq1$, $r_k\geq2$\end{tabular} \\ \cline{2-7} 
		\multicolumn{1}{|l|}{}                           & 7 & $\ell\geq1$, $r_k\geq1$                                           & $\ell\geq1$, $r_k\geq1$                                           & \begin{tabular}[c]{@{}c@{}}$\ell\geq3$, $r_k\geq1$\\ $\ell\geq2$, $r_k\geq2$\end{tabular} & $\ell\geq1$, $r_k\geq1$                                            & \begin{tabular}[c]{@{}c@{}}$\ell\geq3$, $r_k\geq1$\\ $\ell\geq2$, $r_k\geq2$\end{tabular} \\ \hline
	\end{tabular}
\end{table}

\section{Numerical experiments} \label{sec:AOsimulations}

Here we numerically test the performance of the adaptive order $\CWENOZAO$ scheme of order $7$ proposed in Section~\ref{ssec:examples}. This reconstruction will be referred to as \CWZ\ in this section. We compare it with the two $\WENOAO$ reconstructions of order $7$ reviewed in Section~\ref{ssec:wenoAO}. For convenience, they will be referred to as \WAOAHZ\ for the reconstruction of \cite{AHZ18:wao} and \WAOBS\ for the reconstruction of \cite{BGS:wao}.
The study is performed in terms of accuracy and computational cost.

First, in Section~\ref{ssec:recaccuracy} we test the optimal order of convergence of the novel schemes for several choices of the parameters, with the aim of providing a numerical evidence supporting the convergence results of Section~\ref{ssec:cwenozAO}. In Section~\ref{ssec:linear}, the non-oscillatory properties will be analyzed on the linear advection of a non-smooth datum. Next, in Section~\ref{ssec:euler} we consider one-dimensional test problems based on the system of Euler equations for gas dynamics. Finally, in Section~\ref{ssec:balance} we show the performance of the schemes for the solution of a system of balance laws. In Section~\ref{ssec:efficiency} all these experiments are supported by comparisons on the computational cost required by the different schemes. 

For the numerical solution we set up a finite volume scheme of order 7 on a uniform grid, based on the method of lines and the seventh order Runge-Kutta scheme with nine stages~\cite[page 196]{Butcher:2008}. At each Runge-Kutta stage, the cell averages are used to compute the reconstructions and the boundary extrapolated data are fed into the Local Lax-Friedrichs numerical flux. The source term and the initial data are computed with the four point Gaussian quadrature of order seven.
All the simulations are run with a \revision{CFL of $0.9$, unless otherwise stated}.

The $\CWENOZAO$ reconstruction employs 
$\delta_k = \min\left( \Delta x^{r_k}, 0.01\right)$ 
for $k=1,\dots,m$,
$d_1 = 0.15$
and a central optimal weight $d_0=0.85-\sum_{k=1}^m \delta_k$. The values of $\hat{m}$, $\ell$ and $r_k$, $k=1,\dots,m$, are specified in each test. Instead, for the numerical solution of systems of conservation and balance laws we do not perform any tuning on the parameters and we consider $\hat{m}=4$, $\ell=2$ and $r_k=1$, $\forall k$. The $\WENOAO$ reconstructions are implemented as described in~\cite{AHZ18:wao}, and used with the same set of parameters $\hat{m}$ and $\ell$ specified therein.

All computations have been performed with the {\tt claw1dArena} open-source software  \cite{claw1dArena}.

\subsection{Accuracy of the reconstructions}\label{ssec:recaccuracy}
For the accuracy tests we consider the following functions and critical points:
\begin{center}
	\begin{tabular}{cll}
		$n_{cp}$ & \multicolumn{1}{c}{function} & $x_{\text{crit}}$ \\
		0 & $u_0(x)=e^{-x^2}$ & $0.2$\\
		1 & $u_1(x)=\sin(\pi x -\sin(\pi x)/\pi)$ & $0.596683186911209$\\
		2 & $u_2(x)=1.0+\sin^3(\pi x) $ & $0.0$
	\end{tabular}
\end{center}
Obviously, for $n_{cp}=0$, $x_{\text{crit}}=0.2$ is  an evaluation point rather than a critical point.
We compute the reconstruction polynomial for the cell containing the critical point, and the cell averages in the stencil are initialized with the Gauss-Legendre quadrature rule with $4$ nodes. For these tests, quadruple precision has been used. We aim to study numerically the conditions for the optimal convergence discussed in Section~\ref{sssec:analysisSmooth}.

In Table~\ref{tab:recErrorsA} and Table~\ref{tab:recErrorsB} we show reconstruction errors and convergence rates for the parameter values $\hat{m}=4$, $\ell=2$ and $r_k=1$, $\forall k$, and which are used in the following numerical experiments, on critical points of order $n_{cp}=0$ and $n_{cp}=1$, respectively. The optimal order of reconstruction is reached already on coarse grids, and we observe a similar behavior for any $n_{cp}\geq2$.

In Table~\ref{tab:recErrorsC} and Table~\ref{tab:recErrorsD} we focus on $n_\text{cp}=2$, since this value induces a restriction on the parameters $\ell$ and $r_k$ when $\hat{m}=6,7$, see Table~\ref{tab:conditions7}. In particular, in Table~\ref{tab:recErrorsC} we consider a set of parameters, $\hat{m}=6$, $\ell=1$ and $r_k=1$, which does not satisfy the sufficient conditions for optimal convergence. We observe that the order of convergence degrades to order $6$. Conversely, in Table~\ref{tab:recErrorsD} the set of parameters, $\hat{m}=6$, $\ell=1$ and $r_k=2$, satisfy the sufficient conditions for optimal convergence and the optimal expected order is reached already on coarse grids.

This analysis shows that, although the conditions derived in Section~\ref{sssec:analysisSmooth} are only sufficient, they are rather strict since optimal order is degraded when these conditions are not fulfilled.

\begin{table}[t!]
	\caption{Reconstruction errors and order of convergence for $\CWZ$.\label{tab:recErrors}}
	\centering
	\subfloat[$\hat{m}=4, \ell=2, r=1$ on $n_{cp}=0$.\label{tab:recErrorsA}]{
		\pgfplotstabletypeset[
		col sep=comma,
		sci zerofill,
		empty cells with={--},
		every head row/.style={before row=\toprule,after row=\midrule},
		every last row/.style={after row=\bottomrule},
		create on use/rate/.style={create col/dyadic refinement rate={1}},
		columns/0/.style={column name={$\Delta x$}},
		columns/1/.style={column name={error}},
		columns/rate/.style={fixed zerofill},
		columns={0,1,rate},
		skip rows between index={8}{13}
		]
		{cwz753_r1_a2_m4_cp0.err}
 	}
	\subfloat[$\hat{m}=4, \ell=2, r=1$ on $n_{cp}=1$.\label{tab:recErrorsB}]{
		\pgfplotstabletypeset[
		col sep=comma,
		sci zerofill,
		empty cells with={--},
		every head row/.style={before row=\toprule,after row=\midrule},
		every last row/.style={after row=\bottomrule},
		create on use/rate/.style={create col/dyadic refinement rate={1}},
		columns/0/.style={column name={$\Delta x$}},
		columns/1/.style={column name={error}},
		columns/rate/.style={fixed zerofill},
		columns={0,1,rate},
		skip rows between index={8}{13}
		]
		{cwz753_r1_a2_m4_cp1.err}
	}
	\\
	\subfloat[$\hat{m}=6, \ell=1, r=1$ on $n_{cp}=2$.\label{tab:recErrorsC}]{
		\pgfplotstabletypeset[
		col sep=comma,
		sci zerofill,
		empty cells with={--},
		every head row/.style={before row=\toprule,after row=\midrule},
		every last row/.style={after row=\bottomrule},
		create on use/rate/.style={create col/dyadic refinement rate={1}},
		columns/0/.style={column name={$\Delta x$}},
		columns/1/.style={column name={error}},
		columns/rate/.style={fixed zerofill},
		columns={0,1,rate},
		skip rows between index={8}{14}
		]
		{cwz753_r1_a1_m6_cp2.err}
	}
	\subfloat[$\hat{m}=6, \ell=1, r=2$ on $n_{cp}=2$.\label{tab:recErrorsD}]{
		\pgfplotstabletypeset[
		col sep=comma,
		sci zerofill,
		empty cells with={--},
		every head row/.style={before row=\toprule,after row=\midrule},
		every last row/.style={after row=\bottomrule},
		create on use/rate/.style={create col/dyadic refinement rate={1}},
		columns/0/.style={column name={$\Delta x$}},
		columns/1/.style={column name={error}},
		columns/rate/.style={fixed zerofill},
		columns={0,1,rate},
		skip rows between index={8}{14}
		]
		{cwz753_r2_a1_m6_cp2.err}
	}
\end{table}

\subsection{Linear transport problem: Jiang-Shu test} \label{ssec:linear}

We solve the linear scalar conservation law
\begin{equation*} \label{eq:linadv}
u_t+u_x=0
\end{equation*}
on the periodic domain $x\in[-1,1]$ and up to final time $T=8$. As initial condition we consider the non-smooth profile
\begin{subequations}\label{eq:jiangshu}
	\begin{equation} 
	u_0(x) =
	\begin{cases}
	\frac16 \left( G(x,\beta,z-\delta) + G(x,\beta,z+\delta) + 4G(x,\beta,z) \right), & -0.8 \leq x \leq -0.6,\\
	1, & -0.4 \leq x \leq -0.2,\\
	1-\left|10(x-0.1)\right|, & 0 \leq x \leq 0.2,\\
	\frac16 \left( F(x,\alpha,a-\delta) + F(x,\alpha,a+\delta) + 4F(x,\alpha,a) \right), & 0.4 \leq x \leq 0.6,\\
	0, & \text{otherwise}		
	\end{cases}
	\end{equation}
	where
	\begin{equation} 
	G(x,\beta,z) = \exp(-\beta(x-z)^2), \quad F(x,\alpha,a) = \sqrt{\max\{1 - \alpha^2(x - a)^2,0\}}
	\end{equation}
\end{subequations}
and the constants are taken as $a = 0.5$, $z = -0.7$, $\delta = 0.005$, $\alpha = 10$ and $\beta = \log 2/36 \delta^2$. This problem, designed by Jiang and Shu in~\cite{JiangShu:96}, is used in order to investigate the properties of a scheme to transport different shapes with minimal dissipation and dispersion effects. The initial condition~\eqref{eq:jiangshu} is a combination of smooth and non-smooth shapes: precisely, from the left to the right side of the domain, we have a Gaussian, a square wave, a sharp triangle wave and a half ellipse.

\begin{figure}[t!]
	\includegraphics[width=\textwidth]{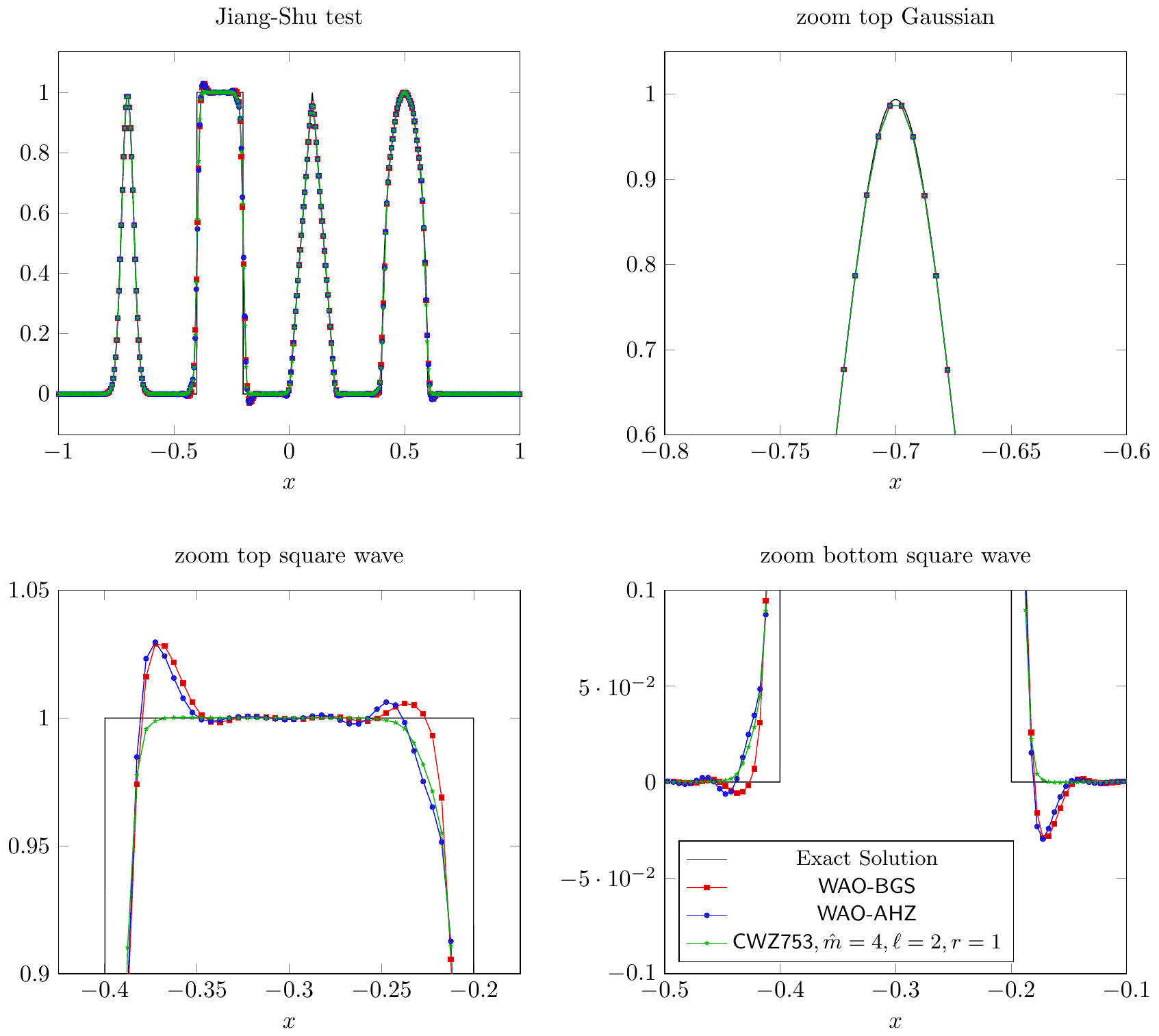}
	\caption{Numerical solution of the Jiang-Shu test problem~\eqref{eq:jiangshu} with schemes of order $7$ and $400$ cells, with zoom on the top part of the Gaussian wave, the top and the bottom part of the square wave.\label{fig:jshu7}}
\end{figure}

Figure~\ref{fig:jshu7} shows the numerical solutions of the Jiang and Shu test problem computed with the $\CWENOZAO$ and the $\WENOAO$ schemes of order $7$ on $400$ cells. In particular, zoom on the top part of the Gaussian wave and of the square wave and zoom on the bottom part of the square wave are considered in order to give information on the behavior of the schemes on smooth and non-smooth zones of the solution. We observe that all the schemes perform similarly and without significant difference on smooth zones. However, compared to $\WENOAO$, the novel adaptive order reconstruction introduced in this work presents less oscillations close to discontinuities. This is made possible by the very small value of $\epsilon$ (large $\hat{m}$) which allows to damp the spurious oscillations. Taking $\hat{m}=3$ we observe small overshoots and undershoots, but their amplitude is still less than the amplitude of the oscillations produced by the $\WENOAO$ schemes. Choices of $\hat{m} < 3$ lead to more oscillating solutions.


\subsection{Euler equations} \label{ssec:euler}

We consider the one-dimensional system of Euler equations for gas dynamics
\[
\pder{}{t} 
\left( \begin{array}{c}
\rho \\ \rho u \\ E
\end{array}\right) +
\pder{}{x} 
 \left( \begin{array}{c}
\rho u \\ \rho u^2 + p \\ u(E+p)
\end{array}\right)  = 0,
\]
where $\rho$, $u$, $p$ and $E$ are the density, velocity, pressure and energy per unit volume of an ideal gas, whose equation of state is
$ E = \frac{p}{\gamma-1} + \frac12 \rho u^2, $
where $\gamma = 1.4$.

\paragraph{Shock-acoustic interaction problem.}
This consists in computing the interaction of a strong shock with an acoustic wave on the domain $x\in[-5,5]$ with free-flow boundary conditions. The problem was introduced by Shu and Osher in~\cite{ShuOsher:89} and is characterized by a Mach 3 shock wave interacting with a standing sinusoidal density wave. The solution, behind the main strong shock, develops a combination of smooth waves and small discontinuities. The initial condition is
$$
(\rho,u,p) = \begin{cases}
(3.857143, 2.629369, 10.333333), & x < -4\\
(1+0.2\sin(5x),0,1), & x \geq -4
\end{cases}
$$
and we run the problem up to the final time $T=1.8$.

\begin{figure}[t!]
	\includegraphics[width=\textwidth]{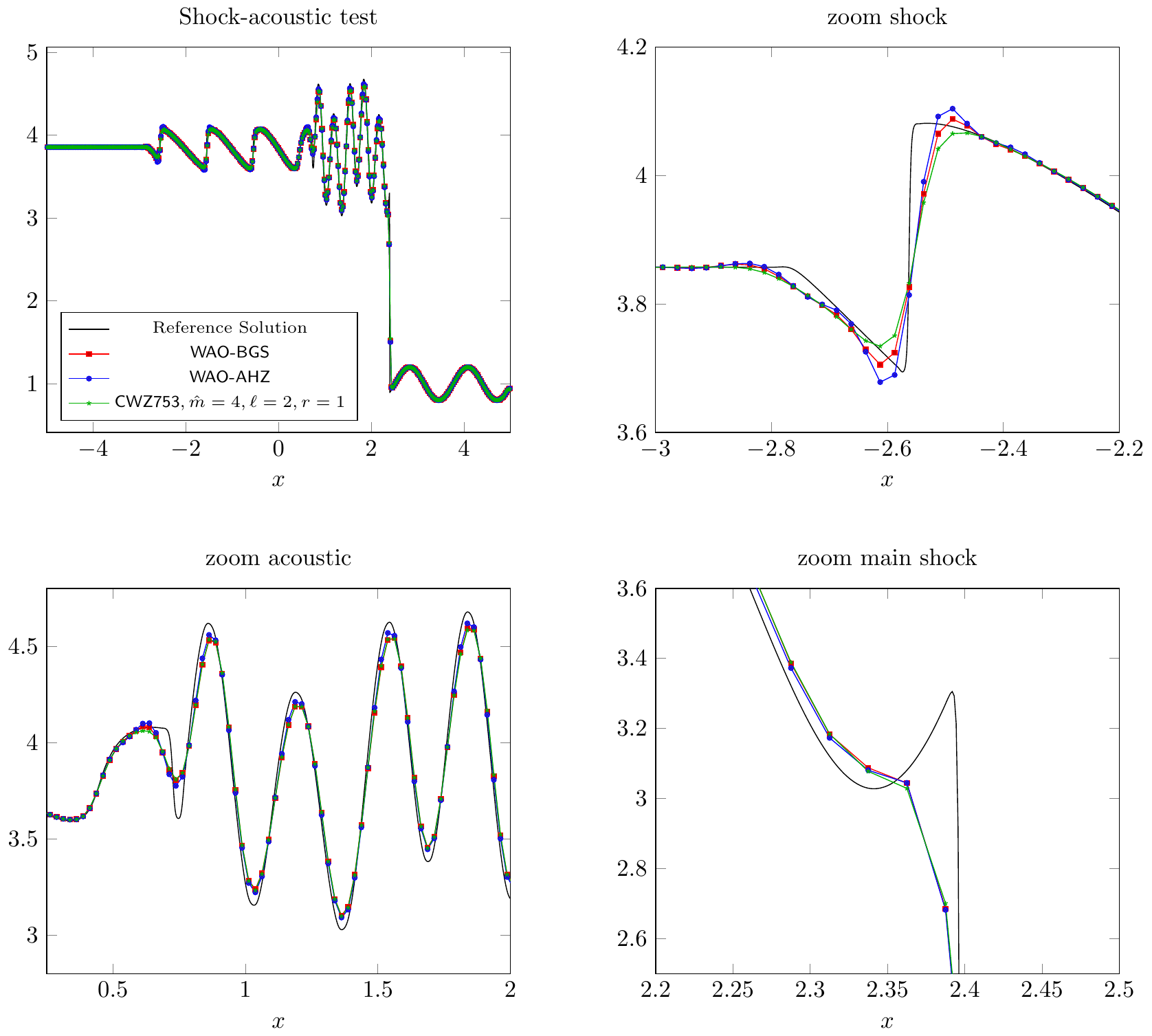}
	\caption{Numerical solution of the shock-acoustic wave interaction problem with schemes of order $7$ on $400$ cells and zoom on three regions of the density profile. Reconstructions are performed along characteristic variables.\label{fig:shockacoustic7}}
\end{figure}

Figure~\ref{fig:shockacoustic7} shows the numerical results computed with the $\CWENOZAO$ and the two $\WENOAO$ schemes of order $7$ on $400$ cells \revision{using a CFL of $0.75$}. We consider the zoom-in of the solution in three regions of the computational domain: the first shocklet, the turbulence zone characterized by the smooth high-frequency solution behind the main shock and the main shock itself. The reference solution (black line) was generated using $8000$ cells and the third order $\CWENO$ scheme. All the reconstructions are computed along characteristic variables. \revision{It should be noted that a local application of the local characteristic projection in few cells around the discontinuities is enough~\cite{Puppo:2003,PuppoSemplice:2011}.}
We observe that all the schemes provide very similar accuracy on the smooth region, where the \WAOAHZ\ scheme  has a slight better resolution close to the extrema. We observe instead different behaviors in the approximation of the first shock in the shocklets region. The \CWZ\ reconstruction is more diffusive, avoiding the small oscillations produced by the $\WENOAO$ schemes. Note that, instead, no extra oscillations at the main shock, which is very underresolved, are observed using any of the schemes.


\paragraph{Lax test.}
We solve the Riemann problem by Lax which is characterized by following initial states:
$$
	(\rho,u,p) = \begin{cases} (0.445,0.6989,3.5277), & x < 0.5 \\ (0.5,0,0.571), & x \geq 0.5 \end{cases}
$$
up to final time $T=0.15$.
The solution develops a rarefaction wave traveling left, a contact discontinuity and a shock, both with positive speeds. This test is challenging when solved by high-order schemes which might produce spurious oscillations on the density peak, between the contact discontinuity and the shock. The oscillations are originated by the interaction between waves in the first stages of the solution, when the discontinuities are so close that the algorithm cannot find a smooth stencil. They can be partly cured computing the reconstruction along characteristic fields, where the waves are approximately decoupled, \cite{QiuShu:02}. Adaptive order reconstructions being able to include low order base levels can also help to reduce oscillations, since smooth stencils can be used. 

\begin{figure}[t!]
	\includegraphics[width=\textwidth]{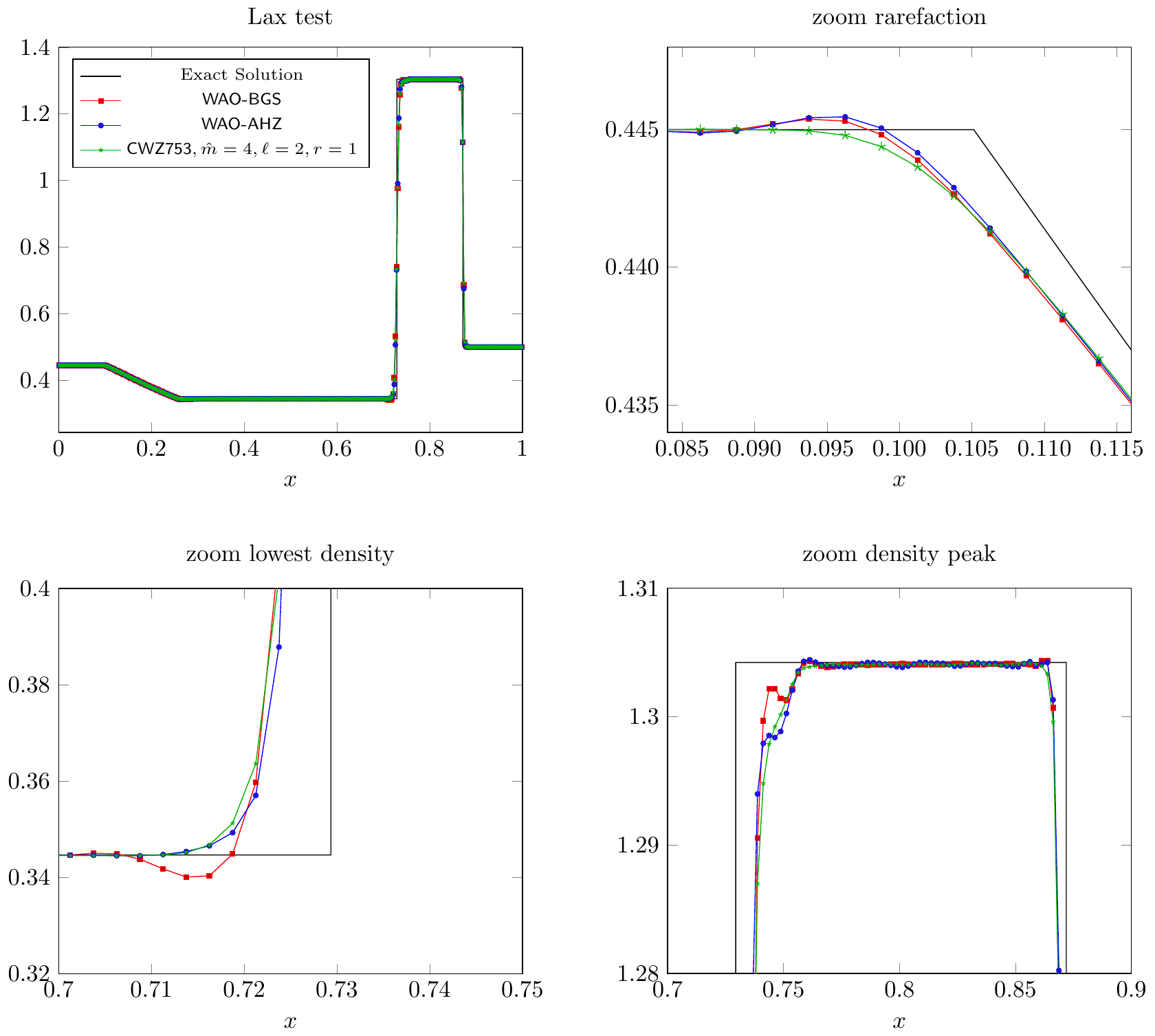}
	\caption{Numerical solution of the Lax test problem with schemes of order $7$ on $400$ cells and zoom on three regions of the density profile. Reconstructions are performed along characteristic variables.\label{fig:lax7}}
\end{figure}

In Figure~\ref{fig:lax7} we show the numerical solution provided by the $\CWENOZAO$ and the $\WENOAO$ schemes of order $7$ on a grid of $N=400$ cells. We consider the zoom-in on the top of the rarefaction wave, on the lowest density point corresponding to the bottom of the contact discontinuity and, finally, on the density peak. The solution of the Riemann problem at final time (black line) is computed exactly~\cite{Toro:book}. All the schemes use reconstructions along characteristic variables.
We observe that the $\WENOAO$ scheme given in~\cite{BGS:wao} produces small oscillations in the three regions, while the $\WENOAO$ scheme given in~\cite{AHZ18:wao} produces a small overshoot only around the top of the rarefaction wave. Instead, the $\CWENOZAO$ reconstructions does not develop spurious oscillations, even if it is less accurate than the $\WENOAO$ scheme given in~\cite{AHZ18:wao} on the bottom part of the contact discontinuity.


\subsection{Balance laws: Euler equations in spherical symmetry} \label{ssec:balance}

In the case of radial symmetry, the multi-dimensional gas dynamics equations can be written as a one-dimensional system, with a source term, which takes into account the geometrical effect, \cite[\S1.6.3]{Toro:book}.
In radial symmetry all the variable are functions of the time $t$ and the radial distance from the origin $\sigma$. Radially symmetric solutions of the Euler equations in $\R^d$ may be computed by solving

\[ \pder{}{t} \left( \begin{array}{c}
\rho \\ \rho u \\ E
\end{array}\right) +
\pder{}{\sigma} \left( \begin{array}{c}
\rho u \\ \rho u^2 + p \\ u(E+p)
\end{array}\right)  
= 
-\frac{d-1}{\sigma} 
\left( \begin{array}{c}
\rho u \\ \rho u^2 \\ up
\end{array}\right)
\]
where $\rho$, $u$, $p$ and $E$ are density, radial velocity, pressure and energy per unit volume of an ideal gas, whose equation of state is still specified by
$ E = \frac{p}{\gamma-1} + \frac12 \rho u^2, $
with $\gamma = 1.4$.  When $d=2$ we have cylindrical symmetry, an approximation to two-dimensional flow. When $d=3$ we have spherical symmetry, an approximation to three-dimensional flow.

We solve the so-called ``explosion problem'' in three space dimensions, which has a shock tube like initial data. In our case, we take Sod's test data, namely
$$
	(\rho,u,p) = \begin{cases} (1,0,1), & \sigma<0.5\\ 
								(0.125,0,0.1), & \sigma>0.5.
				\end{cases}
$$
The final time of the simulation is \minor{$T=0.5$}.
We compute the solution for $\sigma\in[0,1]$ with wall boundary conditions. The Gaussian quadrature formulas of order 7 is employed to compute the cell average of the source term, which also avoids quadrature nodes at $\sigma=0$.

\begin{figure}[t!]
	\includegraphics[width=\textwidth]{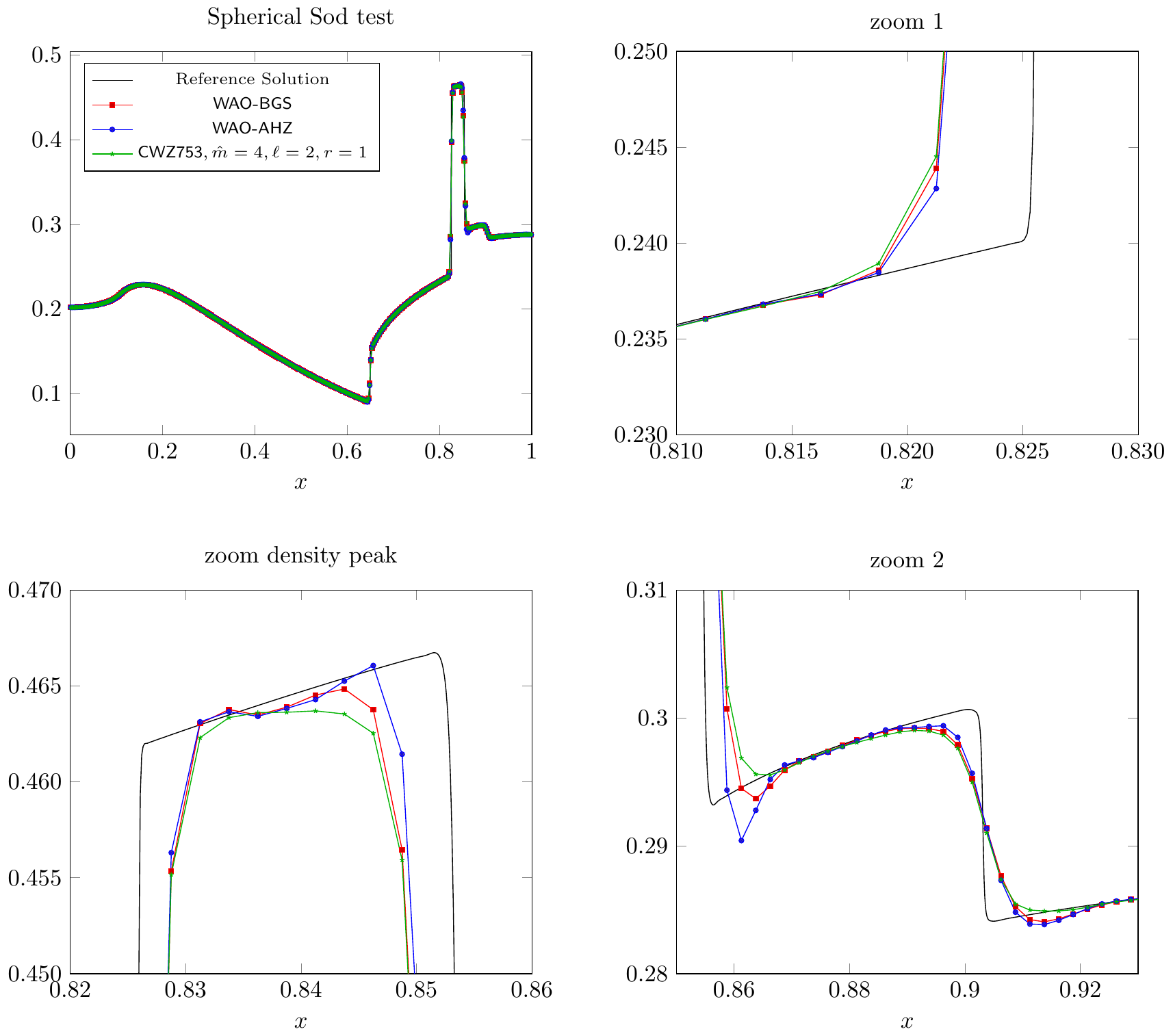}
	\caption{Numerical solution of the radially spherical Sod's explosion problem with schemes of order $7$ with $400$ cells and zoom on three regions of the density profile. Reconstructions are performed along characteristic variables.\label{fig:sodwall7}}
\end{figure}

The density profile
for $d=3$ at final time obtained with $N=400$ cells using the $\CWENOZAO$ and the two $\WENOAO$ schemes is shown in Figure~\ref{fig:sodwall7}, restricted to the domain $\sigma\in[0,1]$. The zoom in the density profile are centered on the bottom part of the discontinuity before the density peak, on the density peak and on the wave between the density peak and the right boundary. All the reconstructions are computed along characteristic variables in order to avoid spurious oscillations. We observe that the $\CWENOZAO$ scheme is more diffusive in the approximation of the discontinuities around the density peak. This is due to the choice of $\hat{m}=4$ and we recall that we do not perform tuning on the parameters. Actually, in this region of the density profile we observe a better accuracy of the $\CWENOZAO$ reconstruction when using $\hat{m}=2$. However, with $\hat{m}=4$ we can avoid the undershoots exhibited by the $\WENOAO$ reconstructions, see the right most panel of Figure~\ref{fig:sodwall7}.


\subsection{Computational efficiency} \label{ssec:efficiency}

All the numerical experiments have shown that the adaptive order $\CWENOZAO$ scheme of order $7$ performs similarly to the two $\WENOAO$ schemes of order $7$ in terms of accuracy and that their parameters may often be tuned to obtain less oscillatory reconstructions. This is not surprising since these schemes share the same idea of nonlinearly blending reconstruction polynomials which cover very high order gap. However, as already discussed in Section~\ref{sec:AOreconstruction}, the $\CWENOZAO$ scheme is defined in such a way definition of intermediate reconstructions is not needed. This make the $\CWENOZAO$ scheme more efficient in terms of computational cost.

In order to compare the computational efficiency of the adaptive order reconstructions, we compiled {\tt claw1dArena} in release mode, which corresponds to {\tt -O3} optimization level of the Gnu C++ compiler. We measured the computational time required by each scheme to solve the numerical experiments proposed in the previous sections, by using three grids, precisely $N=200,400,800$ cells\revision{, and a CFL of $0.45$}. We repeated this step five times and then we computed the median of the CPU times. 
Table~\ref{tab:cputime} contains the results obtained on a quadcore Intel Core i7-6600U with clock speed 2.60GHz (left) and on a dualcore Intel Core i3-2100T with clock speed 2.50GHz (right).

Table~\ref{tab:cputime} shows the CPU times in seconds for the $\CWENOZAO$ reconstruction. Instead, for the $\WENOAO$ schemes we provide information on the difference in percentage compared to time required by the $\CWENOZAO$ scheme. The CPU times for the gas-dynamics problems of Shu-Osher and spherical Sod are computed without employing reconstruction along characteristic variables. Instead, the CPU times for the Lax test are computed by using reconstruction along characteristic variables. We observe that, as we expected, although the accuracy of the schemes is comparable, the major difference is given by their computational cost. In fact, both $\WENOAO$ reconstructions require a larger computational time and therefore there is an increasing advantage in using the $\CWENOZAO$ type reconstruction.


\begin{table}
	\caption{Comparison of the computational times.}
	\label{tab:cputime}
	\centering
		\subfloat[Jiang-Shu test. \label{tab:timejshu} ]{
		\scriptsize
		\begin{tabular}{c|c|c|c|c|c|c|}
			\cline{2-7}
			&\multicolumn{3}{c|}{Core i7-6600U @ 2.60GHz}
			&\multicolumn{3}{c|}{Core i3-2100T @ 2.50GHz}
			\\ \hline
			\multicolumn{1}{|c|}{Cells} & \CWZ & \WAOBS & \WAOAHZ 
			& \CWZ & \WAOBS & \WAOAHZ \\ \hline
			\multicolumn{1}{|c|}{200}   
			& 9.987 s     & +7.87\%   & +11.15\%  
			& 14.39 s    & +9.95\%   & +13.00\%  \\ \hline
			\multicolumn{1}{|c|}{400}   
			& 38.45 s     & +9.25\%   & +11.93\%  
			& 57.16 s    & +10.07\%   & +12.84\%
			\\ \hline
			\multicolumn{1}{|c|}{800}   
			& 153 s       & +8.93\%   & +11.97\%  
			& 229.2 s    & +9.80\%   & +12.28\%  \\ \hline
		\end{tabular}
	}
\\
	\subfloat[Shu-Osher test with reconstruction along conservative variables. \label{tab:timeshockacoustic}]{
		\scriptsize
		\begin{tabular}{c|c|c|c|c|c|c|}
			\cline{2-7}
			&\multicolumn{3}{c|}{Core i7-6600U @ 2.60GHz}
			&\multicolumn{3}{c|}{Core i3-2100T @ 2.50GHz}
			\\ \hline
			\multicolumn{1}{|c|}{Cells} & \CWZ & \WAOBS & \WAOAHZ 
			& \CWZ & \WAOBS & \WAOAHZ \\ \hline
			\multicolumn{1}{|c|}{200}   
			& 3.06 s     & +10.30\%   & +17.29\%  
			& 4.094 s    & +11.13\%   & +18.70\%  \\ \hline
			\multicolumn{1}{|c|}{400}   
			& 12.42 s     & +10.52\%   & +17.75\%  
			& 16.54 s    & +11.27\%   & +18.76\%
			\\ \hline
			\multicolumn{1}{|c|}{800}   
			& 49.09 s       & +9.89\%   & +15.64\%  
			& 66.79 s    & +10.22\%   & +17.56\%  \\ \hline
		\end{tabular}
	}
	\\
	\subfloat[Lax test with reconstruction along characteristic variables. \label{tab:timelax}]{
		\scriptsize
		\begin{tabular}{c|c|c|c|c|c|c|}
			\cline{2-7}
			&\multicolumn{3}{c|}{Core i7-6600U @ 2.60GHz}
			&\multicolumn{3}{c|}{Core i3-2100T @ 2.50GHz}
			\\ \hline
			\multicolumn{1}{|c|}{Cells} & \CWZ & \WAOBS & \WAOAHZ 
			& \CWZ & \WAOBS & \WAOAHZ \\ \hline
			\multicolumn{1}{|c|}{200}   
			& 3.108 s     & +10.15\%   & +16.04\%  
			& 10.82 s    & +9.61\%   & +11.67\%  \\ \hline
			\multicolumn{1}{|c|}{400}   
			& 12.11 s     & +13.81\%   & +15.31\%  
			& 43 s    & +9.00\%   & +10.32\%
			\\ \hline
			\multicolumn{1}{|c|}{800}   
			& 47.92 s       & +13.16\%   & +19.70\%  
			& 172.2 s    & +9.22\%   & +9.93\%  \\ \hline
		\end{tabular}
	}
	\\
	\subfloat[Spherical Sod test with reconstruction along conservative variables. \label{tab:timesod}]{
		\scriptsize
		\begin{tabular}{c|c|c|c|c|c|c|}
			\cline{2-7}
			&\multicolumn{3}{c|}{Core i7-6600U @ 2.60GHz}
			&\multicolumn{3}{c|}{Core i3-2100T @ 2.50GHz}
			\\ \hline
			\multicolumn{1}{|c|}{Cells} & \CWZ & \WAOBS & \WAOAHZ 
			& \CWZ & \WAOBS & \WAOAHZ \\ \hline
			\multicolumn{1}{|c|}{200}   
			& 4.352 s     & +14.14\%   & +17.32\%  
			& 5.332 s    & +11.52\%   & +19.70\%  \\ \hline
			\multicolumn{1}{|c|}{400}   
			& 17.05 s     & +9.16\%   & +15.88\%  
			& 21.21 s    & +10.62\%   & +18.67\%
			\\ \hline
			\multicolumn{1}{|c|}{800}   
			& 65.22 s       & +12.74\%   & +18.72\%  
			& 84.66 s    & +10.78\%   & +18.36\%  \\ \hline
		\end{tabular}
	}
\end{table}

\section{Conclusion} \label{sec:AOconclusion}

In this paper we have presented a novel approach to adaptive order essentially non-oscillatory reconstruction. Our technique relies on the \CWENOZ\ reconstruction with an optimal polynomial of degree $G$ and, in order to preserve the optimal accuracy on smooth data, the candidate polynomials are split into two families. The first family includes all polynomials with degree at least $G/2$ and these are given $\Ogrande(1)$ linear weights, as in the usual weighted essentially non-oscillatory reconstructions. The second family is composed of those candidate polynomials with degree lower than $G/2$, which would lower the accuracy on smooth data if given $\Ogrande(1)$ linear weights. These latter are thus associated to infinitesimal linear weights, proportional to $\DX^{r}$ with $r$ larger than the gap between $G/2$ and the degree of the polynomial.

The main result of this paper is the analysis of the reconstruction, which gives sufficient conditions on the reconstruction parameter that guarantee the optimal convergence rates on smooth data.

As an application we have constructed and tested a $\CWZ$ that mimics the accuracy of $\WENOAO(7,5,3)$ reconstructions already present in the literature. 
The novel reconstruction shows similar accuracy to $\WENOAO(7,5,3)$ of \cite{BGS:wao,AHZ18:wao} on smooth data and, in some numerical tests, slightly reduces the onset of the spurious oscillations.
Since our approach computes the set of nonlinear weights in one go, without resorting to iterative or hierarchic constructions, interesting savings in computational time could be demonstrated.

\revision{We point out that Proposition~\ref{th:proposition}, as its analogous result in~\cite{CSV19:cwenoz}, is valid independently on the type of grid and number of space dimensions. Therefore, any choice of $\tau$ as in~\eqref{eq:tauAO} can be employed to define $\CWENOZAO$ in much more general situations than the ones described in this paper. In particular, we expect that the present approach can yield very fast and accurate reconstructions in more than one space dimensions, since its CPU time savings could be combined with the avoidance of dimensional splitting that is made possible by the \CWENOZ\ approach.} 

Even though we have considered the finite volume formulation, we believe that the same results extend straightforwardly to the finite difference case. Furthermore, also the recently proposed
Multiresolution \WENO\ schemes
\cite{ZhuShu:18:MRWENO,ZhuShu:19:MRWENOtri}, which
are based on a hierarchical computation for the nonlinear weights,  might be amenable to be sped up following the ideas of this paper.

\bibliographystyle{plain}

\bibliography{CSVCWENOZbiblio} 

\end{document}